\documentclass[11pt]{amsart}

\pdfoutput=1

\usepackage{parskip}

\usepackage{fullpage}

\usepackage{amsfonts}
\usepackage{amssymb}
\usepackage{amsmath}
\usepackage{amsthm}
\usepackage{amsbsy}
\usepackage{amssymb} 
\usepackage{verbatim}
\usepackage{bm} 
\usepackage{paralist} 
 \usepackage{color}
 \usepackage{mathrsfs}
 \usepackage{graphicx}
\usepackage{hyperref}
\usepackage{comment}

\makeatletter
\def\thm@space@setup{%
  \thm@preskip=\parskip \thm@postskip=0pt
}
\makeatother

\newcommand{\E}{\mathbf{E}}

\newcommand{\dist}{\mbox{\rm dist}}

\newcommand{\EE}[1]{\E\l[#1\r]}

\renewcommand\Re{\operatorname{Re}}
\renewcommand\Im{\operatorname{Im}}

\def\intt{\int\limits}

\def\to{\rightarrow}

\def\mb{\mbox}

\newcommand{\para}[1]{\vspace{4mm}\noindent{\bfseries #1:}}

\def\l{\left}
\def\r{\right}
\def\<{\langle}
\def\>{\rangle}

\newcommand{\ba}{\[\begin{aligned}}
\newcommand{\ea}{\end{aligned}\]}
\newcommand\mnote[1]{} 
\newcommand{\beq}[1]{\begin{equation}\label{#1}}
\newcommand\eeq{\end{equation}}
\newcommand\ben{\begin{equation}}
\newcommand\een{\end{equation}}
\newcommand\bes{\begin{eqnarray*}}
\newcommand\ees{\end{eqnarray*}}
\newcommand\besn{\begin{eqnarray}}
\newcommand\eesn{\end{eqnarray}}

\def\bthm{\begin{theorem}}
\def\ethm{\end{theorem}}
\def\bdefn{\begin{definition}}
\def\edefn{\end{definition}}
\newcommand{\benu}{\begin{enumerate}\setlength\itemsep{6pt}}
\newcommand{\beit}{\begin{itemize}\setlength\itemsep{3pt}}
\def\eenu{\end{enumerate}}
\def\eeit{\end{itemize}}
\def\beds{\begin{description}}
\def\eeds{\end{description}}
\def\bepr{\begin{problem}}
\def\eepr{\end{problem}}

\def\bprf{\begin{proof}}
\def\eprf{\end{proof}}
\def\berk{\begin{remark}}
\def\eerk{\end{remark}}
\def\bex{\begin{exercise}}
\def\eex{\end{exercise}}
\def\beg{\begin{example}}
\def\eeg{\end{example}}

\def\suchthat{{\; : \;}}

\def\PP{{\mathcal P}}

\def\C{\mathbb{C}}
\def\D{\mathbb{D}} 
\def\EE{\mathbb{E}}
\def\N{\mathbb{N}}
\def\bP{\mathbb{P}}
\def\T{\mathbb{T}} 
\def\R{\mathbb{R}}

\def\Z{\mathbb{Z}}

\newcommand{\supp}{\mbox{\textrm supp}}

\newcommand{\sm}{{\raise0.3ex\hbox{$\scriptstyle \setminus$}}}

\newcommand{\e}{\varepsilon}

\def\alp{\alpha}
\def\bet{\beta}

\renewcommand\phi{\varphi}

\theoremstyle{plain} 
    \newtheorem{theorem}{Theorem}
    \newtheorem*{theorem*}{Theorem}
    \newtheorem{lemma}[theorem]{Lemma}
    \newtheorem{proposition}[theorem]{Proposition}
    \newtheorem{corollary}[theorem]{Corollary}
    \newtheorem{claim}[theorem]{Claim}

\theoremstyle{definition} 
    \newtheorem{definition}[theorem]{Definition}

       \newtheorem{problem}[theorem]{\bf }

    \newtheorem{exercise}[theorem]{Exercise}
        \newtheorem{remark}[theorem]{Remark}
    \newtheorem{example}[theorem]{Example}

\renewcommand{\bex}{\indent\begin{exercise}}
\renewcommand\subset{\subseteq}

\newcommand\eps{\varepsilon}

\usepackage{charter}

\hypersetup{
    colorlinks=true, 
    linktoc=all,     
    linkcolor=blue,  
}
\begin{document}

\begin{abstract}
Erd\"os posed in 1940 the extremal problem of studying the minimal area of the lemniscate $\{|p(z)|<1\}$ of a monic polynomial $p$ of degree $n$ all of whose zeros are in the closed unit disc.
In this article, we prove that there exist positive constants $c,C$ independent of the degree $n$ such that 
\[ \dfrac{c}{\log n} \leq \min \text{Area}( \{ |p(z)|<1 \} ) \leq \frac{C}{\log \log n},\] 
improving substantially the previously best known lower bound (due to Pommerenke in 1961) as well as improving the best known upper bound (due to Wagner in 1988). We also study the inradius (radius of the largest inscribed disc); we provide an estimate for the inradius in terms of the area that confirms a 2009 conjecture of Solynin and Williams, and we use this to give a lower bound of order $(n \sqrt{\log n})^{-1}$ on the inradius, addressing a 1958 problem posed by Erd\"os, Herzog, and Piranian (confirming their conjecture up to the logarithmic factor).
In addition to studying the area of $\{|p(z)|<1\}$, we consider other sublevel sets $\{|p(z)|<t\}$, proving both upper and lower bounds of the same order $1/\log \log n$ when $t>1$ and proving power law upper and lower bounds when $0<t<1$.  We also consider the minimal area problem under a more general constraint, namely, replacing the unit disc with a compact set $K$ of unit capacity, where we show that the minimal area converges to zero as $n \rightarrow \infty$ (giving an affirmative answer to another question of Erd\"os, Herzog, Piranian); we also investigate the structure of the area minimizing polynomials, showing that the normalized zero-counting measure converges to the equilibrium measure of $K$ as the degree $n \rightarrow \infty$.
\end{abstract}

\title{On the area of polynomial lemniscates}

\author{Manjunath Krishnapur}
\address{Department of Mathematics, Indian Institute of Science, Bangalore, India-560012}
\email{manju@iisc.ac.in}

\author{Erik Lundberg}
\address{Department of Mathematics and Statistics, Florida Atlantic University, Boca Raton, USA-33431.}
\email{elundber@fau.edu}

\author{Koushik Ramachandran}
\address{Tata Institute of Fundamental Research, Centre for Applicable Mathematics, Bengaluru, India-560065}
\email{koushik@tifrbng.res.in}

\date{}
\maketitle


\section{Introduction}

The study of lemniscates (level sets and sublevel sets of complex polynomials and rational functions) has a long and rich history, and several classically studied extremal problems concern the geometry of lemniscates. 
Let $p(z)$ be a monic complex polynomial of degree $n$, and consider the (filled) lemniscate $\Lambda_p = \{z \in \C : |p(z)| < 1 \}$.  In the current paper, we consider the area and inradius of this set.

P\'{o}lya proved in 1928 \cite{Polya1928} that the area $A(\Lambda_p)$ satisfies the upper bound $A(\Lambda_p) \leq \pi$ where equality holds if and only if $p(z)$ is (a translation and rotation of) $z^n$; we also mention that Cartan showed \cite{Cartan} in the same year that the lemniscate $\Lambda_p$ is contained in the union of at most $n$ discs the sum of whose radii does not exceed $2e$.  In 1940, Erd\"os~\cite{Erdos1940} posed the problem of studying the minimal area under the constraint that the zeros of $p$ are contained in the closed unit disc.  It is necessary to impose some such constraint on the location of zeros, since otherwise the area of $\Lambda_p$ can be made arbitrarily small by simply moving the zeros of $p$ far from each other.  
Erd\"os restated this problem several times throughout his career.\footnote{The version of the problem stated in \cite{Erdos1940} asks about the portion of the lemniscate inside the unit disc with the zeros confined to the unit circle, while subsequent statements of the problem concern the area of the full lemniscate under the relaxed constraint of zeros being in the closed unit disc.  Since Theorem \ref{thm:arealevel1} relates the two constraints, and since the proof of our lower bound estimates the portion inside the unit disc, our results address both versions of the problem.}  In particular, the minimal area problem was included among the list of geometric problems on polynomials posed in the 1958 paper of Erd\"os, Herzog, and Piranian \cite{EHP}, where the authors applied a result of MacLane \cite{MacLane} in order to show (correcting an erroneous statement in \cite{Erdos1940}) that the minimal area converges to zero as $n \rightarrow \infty$.
Perhaps surprisingly, this statement was not made quantitative until 1988, when Wagner constructed, for arbitrary $\delta>0$, examples with area at most a constant times $(\log \log n)^{-\frac{1}{2} + \delta}$, establishing an upper bound on the minimal area.
Well before this, Pommerenke had proved in 1961 \cite{Pomm1961} a lower bound on the inradius (radius of the largest inscribed disc, see the discussion below), showing as a consequence that the minimal area is at least a constant times $1/n^4$.
In problem lists \cite{Erdos1976}, \cite{Erdos1979} published by Erd\"os in the 1970s, he included the minimal area problem along with the conjecture that for every $\eta >0$ the area is bounded below by $n^{-\eta}$ for all $n$ sufficiently large, and when reprising the problem in the 1980s \cite{Erdos1981}, he asked whether the even stronger lower bound of order $(\log n)^{-1}$ holds.
This specific question is included in the version of the minimal area problem that appears on the recently curated Erd\"os problems website \cite{ErdosWebsite}.
Our following main result gives an affirmative answer, improving substantially Pommerenke's estimate, while also improving the upper bound due to Wagner.

\begin{theorem*}\label{thm:mainthmsimple} Let $\mathcal P_n(\overline{\D}) $ denote the set of monic polynomials of degree $n$ having all zeros in the closed unit disc.  Then for $n\ge 3$,
\[ \dfrac{c}{\log n} \leq \inf_{p \in \mathcal P_n(\overline{\D})}  m(\Lambda_p) \leq \frac{C}{\log \log n},\]
where $m(\cdot)$ denotes the two-dimensional Lebesgue measure.
\end{theorem*}

This result will follow from a more refined result (see Theorem \ref{thm:arealevel1}) that relates this problem to the case when all zeros are confined to the unit circle.

We also consider the minimal area for other levels.  For sublevel sets $|p(z)| < 1 + \e$ with $\e>0$ arbitrary, we determine the sharp order of decay, showing that the minimal area is bounded above and below by constants times $(\log \log n)^{-1}$, see Theorem \ref{thm:arealevelmorethan1}.  For sublevel sets $|p(z)| < 1- \e$ the minimal area is much smaller in light of the order $1/n$ area of the so-called Erd\"os lemniscate (corresponding to the polynomial $p(z) = z^n-1$).  We show that the minimal area in this regime is also bounded below by power law decay, see Theorem \ref{thm:arealevellessthan1}.


More generally, another problem posed in \cite{EHP} concerns the case when the zeros are contained in a set of (logarithmic) capacity at least unity.  The authors of \cite{EHP} asked whether it still holds that the minimal area of the lemniscate $|p(z)| < 1$ approaches zero as the degree $n \rightarrow \infty$.  We used a probabilistic method in \cite{KLR} to give an affirmative answer when the capacity is strictly larger than one.  In the current paper we address the case of unit capacity, giving an affirmative answer under a smoothness assumption.

Another problem posed in \cite{EHP} concerns the minimal inradius of $\Lambda_p$ (under the original constraint that the zeros of $p$ are in the closed unit disc).  The case $p(z)=z^n-1$ that produces the so-called Erd\"os lemniscate was conjectured to be extremal for this problem, i.e., the minimal inradius is conjectured to be asymptotically a constant multiple of $1/n$.  Pommerenke proved that the inradius is at least a constant multiple of $1/n^2$.  We improve this to $1/ (n \sqrt{\log n})$ which supports the conjecture of Erd\"os, Herzog, and Piranian (with only the loss of the logarithmic factor), see Theorem \ref{thm:inradius} below.

A problem relating the inradius to the area was posed by Cuenya and Levis \cite{Cuenya} in 2005.  Namely, they conjectured the existence of a constant $C(n)$ depending only on $n$ such that the inradius $\rho(\Lambda_p)$ satisfies the estimate  $\rho(\Lambda_p) \geq C(n) \sqrt{m(\Lambda_p)}$ for all polynomials $p$ of degree $n$.  This was confirmed by Solynin and Williams \cite{Solynin}; while their method of proof did not provide information on how $C(n)$ depends on $n$, they conjectured that the sharp value of $C(n)$ is inversely proportional to $n$.  We confirm this by proving a quantitative version of the estimate with an explicit value of $C(n)$ in the form of a constant divided by $n$, see  Lemma~\ref{lem:inradius} below.  This estimate is sharp up to the constant factor, as shown by the Erd\"os lemniscate.



In addition to our estimates on the minimal area, we investigate the structure of the sequence of minimizing polynomials, showing that their zeros approximately equidistribute on the unit circle (which holds at all levels), and we also show normality of the sequence of minimizers for levels $0<t<1$ and non-normality of the sequence for levels $t>1$ (with the case $t=1$ being an open problem--we note that establishing non-normality for this case would determine the sharp decay rate for the minimal area to be $(\log \log n)^{-1}$ as it is in the case $t>1$).

The case $p(z) = z^n-1$ with zeros at roots of unity that was mentioned above as the conjectured extremal for the minimal inradius has appeared as the conjectured solution to other extremal problems including the problem of determining the maximal length of the (unfilled) lemniscate $\{ |p(z)|=1\}$ of a monic polynomial.  While the maximal length problem is still open, Fryntov and Nazarov \cite{FN} showed that the Erd\"os lemniscate $\{|z^n-1| = 1 \}$ is indeed locally extremal, and as $n \rightarrow \infty$ the maximal length is $2 \pi n + o(n)$, i.e., asymptotic to the length of the Erd\"os lemniscate.  On the other hand, the (filled) Erd\"os lemniscate is not extremal for the minimal area problem (at level $t=1$), since it has area bounded below by a constant independent of $n$.  The extremal polynomial must deviate in a subtle way from having zeros at roots of unity considering that the zeros do equidistribute ``in the bulk''.  We present numerical evidence (see Section \ref{sec:numerics}) that, at least for small $n$, the zeros of the extremal polynomial are located at a subset of the $2n$th roots of unity with some zero(s) of multiplicity.  It is an intriguing open problem to verify this and determine the precise  configuration of zeros 
for all $n$.

The study of lemniscates (level sets and sublevel sets of complex polynomials and rational functions) has a long and rich history with a wide variety of applications including approximation theory \cite{Totik}, \cite{Bishop2025}, holomorphic dynamics \cite[p. 159]{Milnor}, inverse potential problems \cite{Strakhov}, topology of real algebraic curves \cite{Catanese}, \cite{BauerCatanese}, numerical analysis \cite{Tref}, computer vision \cite{EKS}, \cite{Younsi}, \cite{RichardsYounsi}, \cite{Sharon2006}, moving boundary problems of fluid dynamics \cite{KhavLund}, \cite{LundTotik}, \cite{KMPT}, critical sets of planar harmonic mappings \cite{KhavLeeSaez}, \cite{LLtrunc}, which includes critical sets of lensing maps arising in the theory of gravitational lensing \cite[Sec. 15.2.2]{Petters}, and lemniscates have also appeared prominently in the theory and application of conformal mapping \cite{Bell}, \cite{JT}, \cite{GPSS}.
See also the survey \cite{Richards} which elaborates on some of these lines of research.

We also mention that there have been several recent studies on the geometry of random lemniscates \cite{LLlemni}, \cite{LR}, \cite{EHL},  \cite{KaWi}, \cite{KLR}, \cite{KoushikSubhajit}.  A primary source of motivation for those studies is to provide a probabilistic counterpart to the classical extremal problems on lemniscates (e.g., to compare the ``typical'' case to the extremal case), however, as mentioned above, such probabilistic studies have also lead to novel progress on classical deterministic problems by way of indirect probabilistic proofs, see \cite{KLR} and \cite{KoushikSubhajit}.  In the current paper, we utilize such a probabilistic approach to prove one of the key lemmas, see Lemma \ref{lem:probabilistic}.





\para{Notation} We write $C,c$ to mean positive finite constants. They are pure numbers unless it is mentioned that they depend only on such and such parameter(s). We write $a_n\lesssim b_n$ or $b_n\gtrsim a_n$ to mean that $a_n\le C b_n$ for some constant $C$ and large enough $n$. If $a_n\lesssim b_n$ and $b_n\lesssim a_n$, we write $a_n\asymp b_n$.

\section{Results}
For compact $K\subseteq \C$, let $\mathcal{P}_n(K)$ denote the set of monic complex polynomials of degree $n$, having all its zeros in $K$. For a polynomial $p,$ and $t>0$, let $\Lambda_p(t)=\{z\in \C\; : \; |p(z)|\le t\}$ denote the $t$-level lemniscate of $p$. We simply write  $\Lambda_p$ for $\Lambda_p(1)$. For a sequence of polynomials $\{p_n\}$, we write $\Lambda_n$ for $\Lambda_{p_n}$ for ease of notation . For $t>0,$ define
\begin{align*}
    \kappa_n(K,t)=\inf\{m(\Lambda_p(t))\ : \ p\in \mathcal P_n(K)\}
\end{align*}
where $m(\cdot)$ denotes the Lebesgue measure on the complex plane. In what follows, for a polynomial $p$ with multi-set of zeros $Z$ (with each zero  repeated as many times as its  multiplicity), we define the empirical measure of zeros by
\[\mu_p= \frac{1}{\deg(p)}\sum_{a\in Z}\delta_{a}.\] 

\noindent Throughout this paper, we use tools from logarithmic potential theory in the plane. For $K\subset\C$ a non-empty compact set, and $\mu$ a Borel probability measure on $K$, the logarithmic potential of $\mu$ is the function $U_{\mu}:\C\rightarrow [-\infty, \infty)$ defined by 

\[U_{\mu}(z) = \int_{K}\log|z-w|d\mu(w), \hspace{0.05in}z\in\C.\] 

The term capacity in this paper will always refer to logarithmic capacity. For a comprehensive account of potential theory in the plane and its application to complex analysis, we refer the reader to the works \cite{ransford} and \cite{SaffTotikLogpotential}.

\bthm\label{thm:arealevel1} There exist $0<c<C<\infty$ such that for all large enough $n$,
\begin{align*}
\frac{c}{\log n}\le \frac13\kappa_{n(\log n)^4}(\T,1) \le\kappa_n(\overline{\D},1)\le \kappa_n(\T,1) \le \frac{C}{\log \log n}.    
\end{align*}
\ethm

\bthm\label{thm:arealevelmorethan1} For  $t>1$, there exist $0<c<C<\infty$ (depending only on $t$) such that for all large enough $n$,
\begin{align*}
\frac{c}{\log \log n}\le \frac13 \kappa_{n(\log \log n)^4}(\T,t) \le \kappa_n(\overline{\D},t)\le \kappa_n(\T,t) \le \frac{C}{\log \log n}.  
\end{align*}
\ethm

\bthm\label{thm:arealevellessthan1} For  $t\in (0,1)$, there exist $0<c<C<\infty$ (depending only on $t$) such that for all $n\ge 1$,
\begin{align*}
\frac{c}{n^4}\le \kappa_n(\overline{\D},t)\le \kappa_n(\T,t) \le \frac{C}{n}.  
\end{align*}
On the other hand, $\kappa_n(\T,t)\ge \frac{c}{n^2\log n}$.
\ethm

\begin{remark}
It seems likely that for each $t>0$ we have $\kappa_n(\overline{\D},t)=\kappa_n(\T,t)$. While we have not been able to show this, the comparisons in Theorems \ref{thm:arealevel1} and \ref{thm:arealevelmorethan1} provide some evidence.
\end{remark}

\begin{remark}
The proofs of Theorems \ref{thm:arealevellessthan1} and \ref{thm:arealevelmorethan1} also provide information for $t$ close to $1$ (considering separately the two cases $t \rightarrow 1-$ and $t \rightarrow 1+$).
Namely with $\e_n = \exp(-(\log n)^M)$, for $t=1-\e_n$, we have the upper bound $\kappa_n(\overline{\D},t) \le \kappa_n(\T,t) \lesssim (\log n)^M/n$, while for $t = 1 + \e_n$, we have the lower bound $\kappa_n(\T,t) \ge \kappa_n(\overline{\D},t) \gtrsim 1/\log \log n$.
\end{remark}

 It is natural to ask what happens if we replace the closed disc by another compact  $K$. If $K$ has capacity less than $1$,  Erd\"{o}s, Herzog, and Piranian~\cite{EHP}  showed that $\kappa_n(K,1)$ is bounded below by a positive constant (that may depend on $K$). Further, they asked whether $\kappa_n(K,1)\to 0$ when $K$ has capacity greater than or equal to $1$.
 If $K$ has capacity strictly larger than $1$ and is  sufficiently smooth, it was shown in \cite{KLR} that $\kappa_n(K,1)\le e^{-cn}$ for some $c>0$  that may depend on $K$. Now we answer the case of compact sets with unit capacity.

 \bthm\label{thm:areageneralcapacity1set}
Let $K$ be the closure of a bounded open set having $C^2$-smooth boundary. Assume that $K$ has capacity $1$. Then, $\inf\limits_{n}\kappa_n(K,1)= 0$.
 \ethm

What does the exact configuration of zeros look like for the area minimizing polynomials?  This is an interesting open question.  It is worth remarking that for $t \geq 1$ even in the case of $\overline{\D}$, the natural candidate of $n$-th roots of unity are not minimizers. In fact the corresponding lemniscate areas are bounded away from $0$. However, the upper bounds in Theorems $1-3$ and the following result show that asymptotically the zeros approach the equilibrium measure, hence they approximately equidistribute in the case $K=\overline{\D}$.

\bthm\label{zero distribution}
Let $K\subset\C$ be a compact 
 set with capacity $1$. Let $t>0$ be fixed. Suppose $p_n\in\mathcal{P}_n(K)$ is a sequence such that $m(\Lambda_{p_n}(t)\cap K)\to 0$ as $n\to\infty$. Let $\mu_n$ be the empirical measure of the zeros of $p_n$. Then
\begin{equation}
\mu_n\overset{w}{\to}\nu_{K}, \end{equation}
where $\nu_{K}$ is the equilibrium measure of $K$.
\ethm

Apart from the area of lemniscates, other metric quantities such as the inradius are also of interest. The inradius $\rho(\Omega)$ of an open set $\Omega\subset\C$ is the radius of the largest disc that is completely contained in $\Omega.$ In \cite{EHP} it was asked whether the minimal inradius $\rho_n = \inf\{\rho(\Lambda_p): p\in\mathcal{P}_n(\overline{\D})\}$, satisfies $\rho_n\geq\frac{c}{n}$ for some $c>0.$ Improving upon the bound $\rho_n \geq\dfrac{c}{n^2}$ obtained by Pommerenke, we prove the following result.

\bthm\label{thm:inradius}
The minimal inradius $\rho_n = \inf\{\rho(\Lambda_p): p\in\mathcal{P}_n(\overline{\D})\}$ satisfies
\begin{equation}\label{inradLB}
\rho_n\geq\dfrac{c}{n\sqrt{\log n}}. 
\end{equation}
\ethm

This result follows from Theorem \ref{thm:arealevel1} combined with the following lemma that incidentally confirms a conjecture of Solynin and Williams \cite{Solynin} as mentioned in the introduction.

\begin{lemma}\label{lem:inradius}
Let $t >0.$ Let $p$ be a degree-$n$ polynomial $p$ (not necessarily monic). Then the inradius $\rho(\Lambda_p(t))$ of its associated lemniscate satisfies
\begin{equation}\label{eq:inradArea}
\rho(\Lambda_p(t))\geq \frac{1}{72\pi\sqrt{\pi}} \frac{ \sqrt{m(\Lambda_p(t))}}{n}.
\end{equation}
\end{lemma}




\bigskip

{\color{blue}
}

\section{Overview of main ideas in the proofs}\label{sec:overview}

\noindent In this section we give an overview of the ideas used in some of the proofs of our Theorems. We refer the reader to Section $2$ for the notations. 

\vspace{0.1in}

\subsection{Overview of Lower bounds}
The main tool used to obtain lower bounds on the area of lemnsicates is the following result of Nazarov, Polterovitch and Sodin, see \cite{NPS}. 

\bthm [\cite{NPS}]\label{NPS2}
Let $h: \D\rightarrow \R$ be a non zero harmonic function which is continuous up to the boundary $\T$, satisfying $h(0) = 0.$ Then, there exists a constant $ c_{\mb{\tiny NPS}}>0$ such that

\[m(\{h < 0\}\cap\D) \geq \frac{c_{\mb{\tiny NPS}}}{\log \beta^{*}(\D, h)}\asymp \frac{1}{\log \nu(\T,h)}.\]
Here $\beta^{*}(\D, h) = \max(\beta(\D, h), 3),$ where $\beta(\D, h) = \log\left(\frac{\sup_{\D} |h|}{\sup_{\frac{1}{2}\D}|h|} \right)$ is the so called doubling exponent, and $\nu(\T,h)$ is the number of sign changes of $h$ on $\T$. 
\ethm  

To keep the exposition concise, we will mainly focus on explaining the lower bounds in Theorem \ref{thm:arealevel1}, i.e., at the critical level $t=1$. Let us first start with the case when our polynomial $p\in\mathcal{P}_n(\T)$, i.e. all the zeros of $p$ are on $\T.$ Then $u(z) = \log\vert p(z)\vert$ is harmonic in $\D$, $u(0) = 0$, and importantly for us, $\{\vert p\vert < 1\} = \{u < 0\}$. We apply Theorem \ref{NPS2} to (a rescaling of) $u$, along with the simple fact that $\nu(\T, u)$ is at most $2n$ (see Remark \ref{NPSpolynom}) 
 which yields the lower bound $m(\Lambda_p)\gtrsim\frac{1}{\log n}$ for $p\in\mathcal{P}_n(\T)$. Instead of sign-changes, one can also  estimate the doubling exponent to arrive at the same bound.

 Armed with this, the natural question that arises is the following: Do the lemniscate area minimizers $p_n$ that attain $\kappa_n$, have \emph{all} their zeros on the unit circle? This is an interesting question whose answer seems to be true but we were unable to prove it. Equidistribution, c.f., Theorem \ref{zero distribution}, says that most zeros of $p_n$ are near $\T$ but this still does not rule out $p_n$ having $o(n)$ zeros deep inside $\D$ (say in the half disc). In this case Theorem \ref{NPS2} is not applicable directly since $\log\vert p_n\vert$ is not harmonic in $\D$. To get the comparison of constraints $\frac13\kappa_{n(\log n)^4}(\T,1) \le\kappa_n(\overline{\D},1)$ stated in Theorem \ref{thm:arealevel1}, we prove a \emph{zero-pushing lemma}, see Lemma \ref{pushzeros} (or its alternative version Lemma \ref{lem:probabilistic} proved using a probabilistic approach), which is an important part of our technical contributions in this paper. This lemma says the following: Given a polynomial $p\in\mathcal{P}_n(\overline{\D})$ with at least one zero in the open disc $\D,$ there exists another polynomial $q\in\mathcal{P}_{n(\log n)^4}(\T)$ such that 
\begin{equation}\label{zeropushingheuristic}
U_{\mu_p}(z)=
\frac{1}{n}\log\vert p(z)\vert\leq\frac{1}{n(\log n)^4}\log\vert q(z)\vert
= U_{\mu_q}(z)\hspace{0.1in}\text{for all $z$ in $r\D$}
\end{equation}
where $r=r_n\approx 1$. This is done by showing that for each zero $w$ of $p$ lying ``deep" inside $\D$, we can choose approximately $L=(\log n)^4$ different points $z_1, z_2,..., z_L$ on $\T$ so that \[\log\vert z-w\vert\leq \frac{1}{L}\sum_{j=1}^{L}\log\vert z-z_j\vert\]
holds for all  $z\in\D,$ except in a small set near $\T$. The polynomial $q$ is then built with the different $z_j$ as the zeros (the ``superficial'' zeros near $\T$ can simply be pushed radially to $\T$). It may be helpful to view this replacement procedure as a sort of ``balayage'' of potentials, but we note some essential differences with the standard balayage: besides the need to use a discrete measure, the comparison of the corresponding potentials is within the domain, rather than outside as in the usual balayage process, and moreover, in this situation we can only get a one-sided comparison of potentials (fortunately in the direction we need) as opposed to uniform approximation.  Going back to the proof, the estimate \eqref{zeropushingheuristic} provides the desired comparison of just the portions of $\{\vert q\vert < 1\}$ and $\{\vert p\vert < 1\}$ lying inside the unit disc. We deal with the portion outside the unit disc using a reflection idea from \cite{EHP} while utilizing our equidistribution result to get an improved distortion estimate.

\vspace{0.1in}

\noindent We will be brief about lower bounds for other levels. For $t>1$, the zero-pushing lemma shows that it is sufficient to obtain a lower bound for $\kappa_n(\T, t)$. This in turn is based on estimating the doubling index $\beta(\D , \log\vert p_n\vert)$ for a minimizer $p_n,$ where the crucial point is that the denominator in the expression for $\beta$ is bounded well away from zero (this is closely related to the non-normality of the family of minimizers discussed below), leading to a logarithmic doubling exponent (and hence doubly-logarithmic lower bound on the area). The lower bound for $\kappa_n(\T, t)$ for $t\in (0,1)$ is obtained by estimating the number of sign changes of $\log\vert p\vert$ in balls of small radius, along with a lower estimate for the arc length of $\Lambda(t)\cap\T$ due to Wagner, and certain covering arguments. We refer the reader to Section \ref{alllowerbounds} for detailed proofs. 

\begin{remark}[Zero-pushing with random replacement] We state two slightly different versions of the zero-pushing lemma (see Lemmas \ref{pushzeros} and \ref{lem:probabilistic}).  Either version serves our purpose, but the proof of Lemma \ref{lem:probabilistic} is probabilistic, whereas the proof of Lemma \ref{pushzeros} is constructive, see Remark \ref{rmk:merits} for a discussion of their relative merits.  Here we merely note, from a broad perspective, that the random replacement strategy demonstrates how the ``probabilistic lens'' can be utilized in extremal problems for obtaining a lower bound on an infimum (or upper bound on a supremum) by replacing the original constraint with a more convenient one.
This is done by starting with the unknown extremal and building a random replacement that satisfies the new constraint and compares favorably to the extremal with high probability.
\end{remark}

\subsection{Overview of upper bounds}
As mentioned in the introduction, the best known upper bound for the area of the level-1 lemniscate is due to Wagner~\cite{wagner}, who showed that 
\[
\kappa_n(\D,1)\le \kappa_n(\T,1)\lesssim \frac{1}{(\log \log n)^{\frac12-\delta}}
\]
for every $\delta>0$. Wagner did this by constructing polynomials $p_n$ with zeros on the unit circle, such that $m(\Lambda_{p_n}\cap \D)\lesssim (\log \log n)^{-1+\delta}$. The $1/2$ power appears in considering the full lemniscate, where Wagner invoked an estimate of Erd\"{o}s-Herzog-Piranian comparing the portion outside the unit disc to the portion inside, namely, for any polynomial $p$ with zeros on the unit circle, $m(\Lambda_p)\lesssim \sqrt{m(\Lambda_p\cap \D)}$. This inequality comes from observing that if $re^{i\theta}\in \Lambda_p$ for some $r>1$, then $(2-r)e^{i\theta}\in \Lambda_p$ (every zero of $p$ is closer to the latter point). 
If one applies this reflection idea (rather than the result) together with the geometric observation that in  Wagner's example, $\Lambda_{p_n}\cap \D$ is contained in the union of a small strip around the real axis and a small annulus around the unit circle, then the upper bound immediately improves to $m(\Lambda_p)\lesssim (\log \log n)^{-1+\delta}$.

In addition to removing the square root, our improvement in Theorem~\ref{thm:arealevel1} removes the $\delta$ to get the $(\log \log n)^{-1}$ bound. Secondly, we prove the same upper bound for higher levels. Moreover, we simplify parts of Wagner's proof. We do follow the architecture of Wagner's proof, which is to first construct a finite measure on $\T$ whose potential is positive on most of the disc (but not everywhere, it has to vanish at the origin), and then discretize the measure to get a  potential of the form $\log |p|$ for a polynomial $p$. But in the construction of the measure, we adapt techniques from \cite{NPS} involving rescaling and truncation of a certain entire function. 
This ``top-down'' approach streamlines the construction and ultimately yields a more explicit description of the resulting polynomial that can be estimated directly both inside and outside the unit disc (without need for the reflection idea mentioned above) while also removing the $\delta$ appearing in Wagner's estimate.

The proof of Theorem \ref{thm:areageneralcapacity1set} for unit capacity sets follows this outline, while replacing certain explicit steps with functional analytic arguments.  In the process of adapting the proof to the general setting, we lose quantitative control over the relationship between the area and the degree.

\subsection{Normality of area minimizers}
Fix $t>0$. We will refer to a polynomial $p_{n, t}\in\mathcal{P}_n(\overline{\D})$ which attains the minimum area $\kappa_n(\overline{\D}, t),$ as a $t$-level area minimizer. By Theorems $1-3$, we know that $m(\{\vert p_{n,t}\vert < t\})\to 0$ as $n\to\infty$. Let $\mathcal{F}(t)$ denote the collection of $t$-level minimizers. For which $t>0$ does the family $\mathcal{F}(t)$ form a normal family of holomorphic functions in some neighborhood of the origin? This is not only a question of general interest, but as explained below, sheds important information on the area problem of lemniscates.

\vspace{0.1in}

\noindent It is easily seen that if $t\in (1, \infty),$ then $\mathcal{F}(t)$ is not normal in any neighborhood of the origin. Indeed, if normality holds in some $r\D$, then along a subsequence, we will have $p_{n,t}\to f$, uniformly on compacts of $r\D$, for some holomorphic $f$. Since $\vert p_{n,t}(0)\vert\leq 1,$ we have $\vert f(0)\vert\leq 1$. On the other hand, from $m(\{\vert p_{n,t}\vert < t\})\to 0,$ we conclude that $|f|\geq t > 1$ everywhere in $r\D$. This contradicts $\vert f(0)\vert\leq 1$. In fact the same argument shows that every subsequence $\{p_{n_j, t}\}$ is also not normal in every neighborhood of the origin. As a consequence, $\lim_{n\to\infty}\vert\vert p_{n,t}\vert\vert_{0.49\D}=\infty$, and hence in particular 
\begin{equation}\label{largehalfdiscnorm}
\vert\vert p_{n,t}\vert\vert_{0.49\D}\geq 100    
\end{equation} 
for large enough $n$. By the \emph{zero-pushing lemma} we may assume that $p_{n,t}$ has all zeros on $\T$. Then $\log\vert p_{n,t}\vert$ is harmonic in $\D$. We can now use the trivial upper bound $\vert\vert p_{n,t}\vert\vert_{\T}\leq 2^n$ along with the estimate \eqref{largehalfdiscnorm} to bound the doubling exponent $\beta$ in Theorem \ref{NPS2} and obtain the area lower bound of order $\frac{1}{\log\log n}$.

\vspace{0.1in}

\noindent In contrast we prove that for each fixed $t\in(0, 1)$,  $\mathcal{F}(t)$ forms a normal family in $\D$. This is a consequence of the geometry of the corresponding lemniscates: Using Theorem \ref{NPS2}, we can show that $\Lambda_{p_{n, t}}(t)$ is contained in an annulus near $\T$ of vanishing thickness. In particular, on compact subsets of $\D$, we have $|p_{n,t}(z)|\geq t$ and normality follows.

\vspace{0.1in}

Finally, we discuss the normality of the critical case $t=1$. When it comes to the area, since $\kappa_n(\overline{\D}, 1)$ and $\kappa_{n(\log n)^4}(\T, 1)$ are comparable by Theorem \ref{thm:arealevel1}, it suffices to only consider the normality of the subclass $\mathcal{\tilde F}(1)$ which consist of polynomials $p_n$ which attain $\kappa_n(\T, 1)$, where $n\in\N$. Comparing the geometry of $\Lambda_{p_n}$ with that of the minimizers for $t<1,$ we notice there is an essential difference: By Theorem \ref{NPS2}, we know that $m(\{\vert p_n\vert < 1\}\cap r\D)\geq c_{NPS}r^2\frac{1}{\log\nu(u_n, r\T)}$, where $u_n = \log \vert p_n\vert$. Since the left hand side approaches zero as $n$ tends to infinity, this implies that $\Lambda_{p_n}$ intersect every $r\D$, $r\in(0, 1)$, a growing number of times (as a function of $n$). So the strategy used for $t<1$ does not work here. Unfortunately we are unable to decide whether $\mathcal{\tilde F}(1)$ is normal in $\D$ or not. However see the end of Section \ref{normalsection} for a discussion on the consequences of non-normality of $\mathcal{\tilde F}(1)$ and how it leads to sharp area bounds. We finish that section by relating normality to a certain unstable behavior of the minimizing lemniscates.

\section{Proofs of upper bounds in Theorems~\ref{thm:arealevel1}, \ref{thm:arealevelmorethan1}, \ref{thm:arealevellessthan1} }\label{sec:UB}
In this section, we prove the upper bounds for $\kappa_n(\T,t)$ for all $t>0$. 

For the case $0<t<1$ the upper bound stated in Theorem~\ref{thm:arealevellessthan1} is obtained easily by considering the polynomial $p_n^*(z)=z^n-1$. 
To wit
$$\kappa_n(\T, t)\leq m\left(\Lambda_{p_n^*}(t)\right)\leq \frac{c_0}{n},$$
where $c_0=c_0(t)$ is independent of $n$. The right most inequality follows by observing that if $|z^n - 1| < t,$ then $|z|\in \left( (1-t)^{\frac{1}{n}}, (1+t)^{\frac{1}{n}}\right)$.
Note that $\Lambda_{p_n^*}(1)$ has area bounded away from zero, hence  $p_n^*$ does not give a useful upper bound for $\kappa_n(\T,t)$ for $t\ge 1$.

We now give a construction that shows the stronger statement that
\begin{align}\label{eq:ubdforhighlevel}
\kappa_n(\T,(\log n)^{\alpha})\le \frac{C}{\log \log n}
\end{align}
for some $\alpha>0$ and some $C<\infty$. This of course proves the upper bounds in Theorems~\ref{thm:arealevel1}, \ref{thm:arealevelmorethan1}.

Wagner~\cite{wagner} showed that $\kappa_n(\T,1)\le C_{\delta}(\log \log n)^{-\frac12+\delta}$  for every $\delta>0$. We combine ideas from Wagner's construction along with an ingredient from Nazarov-Polterovich-Sodin~\cite{NPS} to construct a sequence of polynomials that give  the improved estimate~\eqref{eq:ubdforhighlevel}.  
First we describe the ingredient from \cite[Section~6]{NPS} that we use. 

There exists an  entire function $E$ that is real on the real line, and satisfies  \begin{inparaenum}[(a)] 
\item $|E(z)|\le c_1$ for some $c_1$ and all  $z$ in $\{\Re(z)> 0, |\Im z|\ge \frac{\pi}{2}\}$, 
\item $E(x)\sim e^{e^x}$ as $x\to \infty$ on the real axis,  and 
\item if   $Q_R(z)=E(R+z)-E(R)=\sum_{k=1}^{\infty}a_k(R)z^k$ 
then  
\ben\label{eq:coeffest}
|a_k(R)|\le \left(\frac{c}{\log k}\right)^k \mbox{ for all }k\ge 1,
\een
 for some $c<\infty$ and all $R>1$.
\end{inparaenum}

The proof of \eqref{eq:ubdforhighlevel} will be carried out in several steps, whose broad outline is as follows. 
\begin{itemize}
    \item First we construct a harmonic function $u$ on $\overline{\D}$ vanishing at the origin and such that area of $\{u<T\}\cap \D$ is small for a suitably large positive $T$.
\item We show that $u$ can be replaced by $U_{\mu}$, the logarithmic potential of a probability measure $\mu$ on $\T$, such that $\{U_{\mu}<T\}\cap \D(0, 3)$ has small area.
\item We modify $\mu$ to get a discrete measure $\nu$ on $\T$, such that $\{U_{\nu}<T\}\cap\D(0, 3)$ has small area. 
\item Using $\nu$, we construct polynomials $p_n\in\mathcal{P}_n(\T)$ whose $t_n$-lemniscate has small area, for $t_n=(\log n)^{\alpha}$ for some $\alpha>0$.
\end{itemize}

\bigskip

\begin{proof}[Proof of \eqref{eq:ubdforhighlevel}]
{\bf Step-1}: Define truncations of the series defining $Q_R$ by $Q_{R,N}(z)=\sum_{k=1}^{N}a_k(R)z^k$. Make the choice $N=\big\lceil e^{2cR}\big \rceil$ (this will hold throughout the remaining steps). From the coefficient estimate \eqref{eq:coeffest},  we have $|Q_R(z)-Q_{R,N}(z)|\le \sum_{k>N}(c|z|/\log N)^k\le \frac{1}{2^N}$ if $|z|\le R$. Therefore, if $|z|\le R$, then 
\begin{align*}
\Re Q_{R,N}(z)&\le \Re (E(R+z)-E(R))+\frac{1}{2^N} \nonumber \\
&\le |E(R+z)|-E(R)+\frac{1}{2^N}. 
\end{align*}
Now $|E(R+z)|\le c_1$ if  $|\Im z|\ge \frac{\pi}{2}$. Therefore, if  $R$ is large enough, then
\begin{align}\label{eq:bdonrealQRN}
\Re Q_{R,N}(z)&\le -\frac12 E(R), \;\;  \mbox{ if }|z|\le R\mbox{ and } |\Im z|\ge \frac{\pi}{2}.
\end{align}
Now define $u_{R,N}(z)=-\Re Q_{R,N}(Rz/2)$. Then $u_{R,N}$ is harmonic on the plane, $u_{R,N}(0)=0$ and \eqref{eq:bdonrealQRN} implies that $\{u_{R,N}<\frac12 E(R)\}\cap \D(0,2)$ 
has area at most $\frac{C}{R}$ (here $C=8\pi$ suffices).

\medskip
{\bf Step-2}:  Define  $v_{R,N}(z)=1 + \frac{2}{A_{R,N}}\sum_{k=1}^N ka_k(R)(R/2)^k \Re(z^k),$ where  $A_{R,N}=4\sum_{k=1}^Nk|a_k(R)|(R/2)^k$. Then $\frac12 \le v_{R,N}\le \frac32$ on $\T$ and hence $v_{R,N}(e^{i\theta})\frac{d\theta}{2\pi}$ defines a probability measure $\mu=\mu_{R,N}$ on $\T$ with Fourier coefficients $\hat{\mu}(0)=1$ and $\hat{\mu}(k)= \frac{1}{A_{R,N}}ka_k(R)(R/2)^k$ for $1\le |k|\le N$ (and all other Fourier coefficients are zero).

 Observe that the logarithmic potential of $\mu$ is given by
\ba
U_{\mu}(z) &= \int \log|z-e^{it}|d\mu(t) = \begin{cases}
-\sum_{k\ge 1}\frac{\hat{\mu}(k)}{k}\Re(z^k) & \mb{ if }|z|<1\\
\log|z|-\sum_{k\ge 1}\frac{\hat{\mu}(k)}{k}\Re(\frac{1}{z^k}) & \mb{ if }|z|>1, 
\end{cases}
\ea
Plugging in the  Fourier coefficients of $\mu$, we obtain 
\ba
U_{\mu}(z)&=\begin{cases}
-\frac{1}{A_{R,N}}\sum_{k=1}^Na_k(R)\Re((Rz/2)^k) & \mb{ if }|z|<1, \\
\log|z|-\frac{1}{A_{R,N}}\sum_{k=1}^Na_k(R)\Re((R/(2z))^k) & \mb{ if }|z|>1,
\end{cases} \\
&\ge
\begin{cases}
\frac{1}{A_{R,N}}u_{R,N}(z), & \mb{ if }|z|<1, \\
\frac{1}{A_{R,N}}u_{R,N}(1/z) & \mb{ if }|z|>1.
\end{cases}
\ea
Using the conclusion of  Step-1,  we get $U_{\mu}(z)>\frac{E(R)}{2A_{R,N}}$ if either \begin{inparaenum}[(i)]\item $|\Im z|\ge \frac{\pi}{R}$ and $|z|<1$ or \item $|\Im \frac{1}{z}|\ge \frac{\pi}{R}$ and $|z|>1 $. \end{inparaenum} As $\Im \frac{1}{z}=\frac{\Im z}{|z|^2}$, it follows that 
\ben\label{eq:UmuLowerBd}
U_{\mu}(z)>\frac{E(R)}{2A_{R,N}}, \;\; \mbox{ if }|z|\le 3\mbox{ and }|\Im z|\ge \frac{9\pi}{R}. 
\een
In particular, the complementary region $\{U_{\mu}\le \frac{E(R)}{2A_{R,N}}\}\cap \D(0,3)$ has area at most $\frac{C}{R}$.

{\bf Step-3}:  Choose $M=\Big\lceil \frac{16RA_{R,N}}{E(R)}\Big\rceil$ (for reasons to be clear later) and pick angles $\theta_1,\ldots ,\theta_M$ so that the arcs satisfy $\mu([\theta_j,\theta_{j+1}])=1/M$ for all $1\le j\le M$ (here $\theta_{M+1}=\theta_1$). As $\frac12 \le v_{R,N}\le \frac32$, each of these arcs is of length  between $1/(2M)$ and $2/M$. Define the normalized counting measure  $\nu:=\frac{1}{M}\sum_{j=1}^M\delta_{e^{i\theta_j}}$.  Observe  that if $t,s$ belong to the same arc $[\theta_j,\theta_{j+1}]$, then
\[
|\log|z-e^{it}|-\log|z-e^{is}||\le \frac{1}{d(z,S^1)}|e^{it}-e^{is}| \le \frac{4}{|1-|z|| \ M}.
\]
Therefore, taking $s=\theta_j$  and integrating over $t \in [\theta_j,\theta_{j+1}]$ with respect to $d\mu(t)$, and then summing over $j=1,2,...,M$ we get
\ba
|U_{\mu}(z)-U_{\nu}(z)|\le \frac{4R}{M} \;\; \mb{ if }|z|\in [0,3]\setminus [1-R^{-1},1+R^{-1}].
\ea
Now, $\frac{4R}{M}\le \frac{E(R)}{4A_{R,N}}$ by the choice of $M$, hence in view of \eqref{eq:UmuLowerBd} we have $|U_{\mu}-U_{\nu}|<\frac12 U_{\mu}$ on $\{|\Im z|\ge \frac{9\pi}{R}\}\cap \{|z|\in [0,3]\setminus [1-R^{-1},1+R^{-1}]\}$. Therefore, in this region,
\[
U_{\nu}(z)>\frac12 U_{\mu}(z) \ge \frac{E(R)}{4A_{R,N}}\ge \frac{4R}{M}
\]
by the choice of $M$. In particular, $\{U_{\nu}<\frac{4R}{M}\}\cap \D(0,3)$ has area at most  $6\times \frac{18\pi}{R}+\frac{4\pi}{R}\le \frac{C}{R}$. 

 From the estimates \eqref{eq:coeffest} for $a_k(R)$, 
\ba
A_{R,N}\le 4N\sum_{k=1}^N(cR/2)^k \le 4N^2(cR/2)^N \le 4e^{4cR}(cR)^{e^{2cR}}\le e^{e^{CR}}
\ea
where $C>2c$.
For $C$ large, the quantity on the right hand side  dominates $R$ as well as $E(R)\sim e^{e^{R}}$. Hence we  get the  bound $M\le e^{e^{CR}}$ (with a possibly larger $C$) or equivalently $R\ge \frac{1}{C}\log \log M$. Therefore,
\ben\label{eq:Unu}
\mbox{area}\left(\left\{U_{\nu}\le \frac{4\log \log M}{CM}\right\}\cap \D(0,3)\right)\le \frac{C'}{\log \log M}.
\een

{\bf Step-4}: Now we define the polynomial 
\[
p_M(z)=\prod_{k=1}^M(z-e^{i\theta_k}).
\]
Then $\frac{1}{M}\log|p_M(z)|=U_{\nu}(z)$, and hence applying \eqref{eq:Unu} we have (taking $\alpha=\frac{4}{C}$)
\begin{align*}
\mbox{area}\left(\left\{|p_M| \le (\log M)^{\alpha}\right\}\cap \D(0,3)\right)\le \frac{C'}{\log \log M}.
\end{align*}
 As $p_M$ has all zeros on $\T$, we observe that for $|z|\ge 3$,
\[
|p_M(z)|\ge 2^M 
\]
which is larger than $(\log M)^{\alpha}$ for large enough $M$, so that the lemniscate $\Lambda_{p_M}((\log M)^{\alpha})$ is contained in $\D(0,3)$ and has area bounded by $\frac{C'}{\log\log M}$. 

All that is left is to show that $M$ runs through all large enough positive integers. To see this (remember that $N$ depends on $R$ and is integer valued), observe that
\[
\frac{RA_{R,N}}{E(R)} = \sum_{k=1}^N\frac{k|a_k(R)|(R/2)^k}{E(R)}.
\]
Each term in the above sum is a continuous function of $R$. Further, from \eqref{eq:coeffest}, we see that 
\[
\frac{(N+1)|a_{N+1}(R)|(R/2)^{N+1}}{E(R)}\to 0
\]
as $R\to \infty$. Putting these together, it follows that  there exists $M_0\in\N$ such that $M$ attains all integer values in $[M_0,\infty)$ as $R$ increases. 

In conclusion,  for each $M\ge M_0$, we can find a polynomial $p_M\in\mathcal{P}_M(\T)$ which satisfies
\begin{align}\label{eq:finalubdpoly}
    m(\Lambda_{p_M}((\log M)^{\alpha})\le \frac{C'}{ \log \log M},
\end{align}
and this implies the desired upper bound \eqref{eq:ubdforhighlevel}.
\end{proof}

\section{Proof of Theorem \ref{thm:areageneralcapacity1set}}
Let $K=\overline{\Omega}$, where $\Omega$ is a bounded open set with $C^2$ smooth boundary $\partial \Omega$, and moreover $\mbox{cap}(K) = 1$. 

Without loss of generality, we assume that $\Omega$ is simply connected. If not, we can replace $\Omega$ by $\hat{\Omega}$, defined as the union of $\Omega$ with all the bounded components of $\C\setminus \Omega$. Then $\hat{\Omega}$ also has capacity $1$. In addition, the polynomials we construct for $\hat{\Omega}$ will have zeros on $\partial\hat{\Omega}\subseteq \partial \Omega$, hence the same polynomials work for $\Omega$.

Henceforth assume that $\Omega$ is simply connected. Let $\nu$ denote the equilibrium measure of $\bar{\Omega}$. The unit capacity condition along with Frostman's lemma gives that $U_{\nu}=0$ on $\bar{\Omega}$. If $\sigma$ denotes the arc length measure on $\partial \Omega$, we can write $d\nu=gd\sigma$ for some $g\in C^{1}(\partial \Omega)$, with $B_1 < g < B_2$ for some finite positive constants $B_1$ and $B_2$, c.f., \cite[Ch. II, Cor. 4.7]{Marshall}.

The proof is analogous to the proof for the unit disc, with some modifications. 

{\bf Step-1}: Fix $\eps > 0$. The goal is to construct $u\in C(\bar{\Omega})$ that is harmonic in $\Omega$, satisfies $\int ud\nu=0$, and so that $m\{u\le T\}<\eps$ for a suitably large $T>0$. 

 Let $E$ be the  entire function used in Step-1 of the proof of \eqref{eq:ubdforhighlevel}. For $R>0$, and $\alpha\in [0,2\pi)$, define
\[v_{R,\alpha}(z)=\Re E(R+Rze^{i\alpha}), \hspace{0.05in} z\in \overline{\Omega}\]  Observe that  
\ba
\intt_0^{2\pi}\!\!\intt_{\partial \Omega}v_{R,\alpha}d\nu \frac{d\alpha}{2\pi}= \intt_{\partial \Omega}\!\!\intt_0^{2\pi}\Re E(R+Rze^{i\alpha}) \frac{d\alpha}{2\pi} \ d\nu = \intt_{\partial \Omega} \Re E(R) d\nu = E(R).
\ea
Hence, there is some $\alpha$ (depending on $R$) such that $\int_{\partial \Omega}v_{R,\alpha}d\nu\ge  E(R)$. We fix that $\alpha$ and set $u(z)=-v_{R,\alpha}(z)+\int_{\partial \Omega}v_{R,\alpha}d\nu$ for $z\in \overline{\Omega}$ and let $T=\frac12 E(R)$. Then $u$ is harmonic in $\Omega,$ and $\int ud\nu=0$. Further, if $u(z)<T$, then $\Re E(R+Rze^{i\alpha})>\frac12 E(R)>c_1$, if $R$ is large enough. Hence $|\Im\{Rze^{i\alpha}\} | \le \pi$. But this is a strip of width $C/R$, hence 
\begin{align}\label{eq:bddforharmonic}
    m\left(\left\{u<T \right\}\cap \Omega\right)\le \frac{C}{R}.
\end{align}

{\bf Step-2}:
We want to show that there exists a finite positive measure $\mu$ on $\partial \Omega$ such that $m(\{U_{\mu}\le T/2\}\cap \Omega)\le m(\{u\le T\}\cap \Omega)+\eps$.

Fix a  compact $F\subseteq \Omega$ such that $m(\Omega\setminus F)<\eps$. From Lemma~\ref{lem:approxharmonicpotential} below, we obtain a signed measure $\theta$ having bounded density w.r.t. $\sigma$, such that $\|U_{\theta}-u\|_{\sup(F)}<\frac{T}{2}$. Then 
$$
    m\left(\left\{U_{\theta}\le \frac{T}{2}\right\}\cap \Omega\right)\le m\left(\{u\le T\}\cap \Omega\right)+\eps.
$$

Using the fact that the equilibrium measure $\nu$ has a density bounded below, we see that if $C$ is  large enough and $b$ is chosen appropriately,  then $\mu=b(\theta+C\nu)$ is a probability measure with density $h$  (w.r.t. $\sigma$) satisfying $M^{-1}\le h\le M$ for some $M\in (0,\infty)$.  But $U_{\theta}=bU_{\mu}$ on $\bar{\Omega}$, hence
\begin{align}\label{eq:bddforUmu}
m\left(\left\{U_{\mu}\le \frac{T}{2b}\right\}\cap \Omega\right)\le m\left(\{u\le T\}\cap \Omega\right)+\eps.
\end{align}

{\bf Step-3}:  Let $d\mu(t)=h(t)d\sigma(t)$ be the probability measure constructed in the previous step. Discretize $\mu$  by partitioning $\partial \Omega$ into arcs $I_1,\ldots ,I_N$ of equal measure  $\mu(I_j)=\frac{1}{N}$. Choose $w_j\in I_j$ and define  $\mu_N=\frac{1}{N}(\delta_{w_1}+\ldots +\delta_{w_N})$.  Clearly we can couple $V\sim \mu$ and $V_N\sim \mu_N$ such that they belong to the same $I_j$, in particular $|V-V_N|\le \sigma(I_j)\le \frac{M}{N}$ (by the lower bound on $h$). Hence, if $z\in \Omega$ with $d(z,\partial \Omega)\ge \frac{2M}{N}$, then applying the mean-value theorem to $x\mapsto \log|z-x|$, we get
\begin{align*}
    |U_{\mu_{N}}(z)-U_{\mu}(z)| &\le \E[|\log|z-V|-\log|z-V_N||]  \\
    &\le \frac{1}{d(z,\partial \Omega)-\frac{M}{N}}\E[|V_N-V|] \\
    &\le \frac{2M}{Nd(z,\partial \Omega)}.
\end{align*}
Choose $N=\lfloor 2M/\eps^2\rfloor$. Then  $|U_{\mu}-U_{\mu_N}|<2\eps$ on $\Omega_{\eps}=\{z\suchthat d(z,\partial \Omega)>\eps\}$. Therefore, for large enough $R$,
\begin{align}\label{eq:lemniscatebdformuN}
m\left(\left\{U_{\mu_N}\le \frac{T}{4b}\right\}\cap \Omega\right) &\le m\left(\left\{U_{\mu}\le \frac{T}{2b}\right\}\cap \Omega\right)+m(\C \setminus \Omega_{\eps}) \nonumber \\
&\le \frac{C}{R}+C\eps
\end{align}for a large enough constant $C$. In the last line we used  \eqref{eq:bddforharmonic} and \eqref{eq:bddforUmu}. For large $R$, the right side of \eqref{eq:lemniscatebdformuN} is smaller than $2C\eps$.

{\bf Step-4}: Let $P_N= (z-w_1)\ldots (z-w_N)$. Then  $\{|P_N|\le 1\}=\{U_{\mu_N}\le 0\}$ and therefore, 
\begin{align*}
m(\Lambda_{P_N}\cap \Omega)\le m\left(\left\{U_{\mu_N}\le \frac{T}{4b}\right\}\cap \Omega\right) \le 2C\eps.
\end{align*}
To summarize, taking a sequence $\eps_k\downarrow 0$, we have $N_{k}\uparrow \infty$ and polynomials $P_{N_{k}}$ of degree $N_{k}$ such that $m(\Lambda_{P_{N_{k}}}\cap \Omega)\downarrow 0$. From Theorem~\ref{zero distribution}, it follows that the empirical distribution $\mu_{k}$ of the zeros of  $P_{N_{k}}$ converges to $\nu$. By the harmonicity of $U_{\nu}$ on  $\C\setminus \overline{\Omega}$, we know that 
\begin{align}\label{eq:UnupositiveoutsideOmega}
\inf\{U_{\nu}(z)\ : \ d(z,\Omega)>\delta\} >0
\end{align}
for any $\delta>0$. We claim that $U_{\mu_{k}}\to U_{\nu}$ uniformly on $\{z\ : \ d(z,\Omega)>\delta\}$. To see this, 
couple $V_{k}\sim \mu_{k}$ with $V\sim \nu$ so that $V_{k}\to V$ a.s., and hence $\E[|V_k-V|]\to 0$ by the dominated convergence Theorem. Therefore,
\begin{align*}
    |U_{\mu_{k}}(z)-U_{\nu}(z)| &\le \E[|\log|z-V|-\log|z-V_{k}||] \\
    &\le \frac{1}{d(z,\partial \Omega)}\E[|V_{k}-V|]
\end{align*}
which converges to $0$ as $k\to \infty$. Then \eqref{eq:UnupositiveoutsideOmega} implies that 
\[
m(\Lambda_{P_{N_k}}) \le m(\Lambda_{P_{N_{k}}}\cap \Omega) + C\delta.
\]
Let $k\to \infty$ and $\delta\downarrow 0$, we see that $\inf_n \kappa_n(K,1)=0$. 

This completes the proof of Theorem~\ref{thm:areageneralcapacity1set} except for the following lemma.

\begin{lemma}\label{lem:approxharmonicpotential} Let $\Omega\subseteq \C$ be a bounded domain with $C^2$ boundary. Let $\nu$ be the equilibrium measure of $\overline{\Omega}$ and let $\sigma$ be the arc-length measure on $\partial \Omega$. Let $u\in C(\bar{\Omega})$ be harmonic in $\Omega$ and satisfy $\int ud\nu=0$. Then there exist signed measures $\theta_n$ on $\partial \Omega$ having continuous densities w.r.t. $\sigma$ and such that $U_{\theta_n}\to u$ uniformly on compact subsets of $\Omega$. 
\end{lemma}
\begin{proof} Let $\phi_w(z)=\log|z-w|$ and let $\mathcal{H}=\overline{\mb{span}}\{\phi_w\suchthat w\in  \partial \Omega\}$ inside $L^2(\sigma)$. Let $g$ be the density of $\nu$ w.r.t. $\sigma$ (note that $g$ is supported on the exterior boundary of $\Omega$). Then,  $\<\phi_w,g\>_{L^2(\sigma)}=U_{\nu}(w)=0$ for all $w\in \Omega$, hence $g\perp\mathcal{H}$. 

We claim that if $f\perp \mathcal{H}$, then  $f=cg$ for some $c\in \R$. To prove this, let $f\perp\mathcal{H}$. Then 
\begin{align}\label{eq:signedpotentialvanishing}
U_{fd\sigma}(w)=\<f,\phi_w\>=0 \hspace{0.05in}\mbox{for each}\hspace{0.05in} w\in  \partial \Omega.    
\end{align}
Write $f=f_+-f_-$ and let $c_{\pm}=\int f_{\pm }d\sigma$. Without loss of generality we may assume that $c_+\ge c_-$. Now set $d\mu_1=f_+d\sigma$ and $d\mu_2=f_-d\sigma+(c_+-c_-)gd\sigma$. Then both $\mu_1$ and $\mu_2$ are positive measures with the same total mass and absolutely continuous to $\sigma$. Using \eqref{eq:signedpotentialvanishing} along with the fact that $U_{\nu} = 0$ on $\partial\Omega,$ we obtain $U_{\mu_1}= U_{\mu_2}$ in $\partial \Omega$. Further $U_{\mu_1}(z)-U_{\mu_2}(z)\to 0$ as $z\to \infty$, as $\mu_i$ have the same total mass. Moreover, $U_{\mu_i}$ are continuous on $\C$ as $\mu_i$ are absolutely continuous\footnote{More generally, suppose $\mu$ is a finite (positive) measure on $\partial \Omega$ such that $\mu(B(w,r))\le Cr^{\alpha}$ for all $w$ and $r>0$, for some $\alpha>0$ and  $C<\infty$. If $w_n\to w$, then it follows that $U_{\mu}(w_n)\to U_{\mu}(w)$, provided we show that $f_n(z)=\log|z-w_n|$ are uniformly integrable w.r.t. $\mu$. To see this, note that  $(f_n)_+$ are obviously uniformly bounded and $(f_n)_-$ are bounded in $L^2(\mu)$. The latter follows from writing $\int (f_n)_-^2d\mu$ as   \begin{align*}
        \int_0^{\infty} \mu\{z\ : \ (\log_-|z-w_n|)^2\ge t\}dt        &=\int_0^{\infty}\mu(B(w_n,e^{-\sqrt{t})}))dt \ \le \ C \int_0^{\infty}e^{-\alpha \sqrt{t}}dt.
    \end{align*}} to $\sigma$. Therefore, by the maximum principle (applied to the two components of the complement of the exterior boundary of $\Omega$) we get  $U_{\mu_1}=U_{\mu_2}$ on $\C$. This implies (on applying Laplacian)  that $\mu_1=\mu_2$. Hence, $f_+=f_-+(c_+-c_-)g$ or in other words, $f=cg$ for some $c$.

Now suppose $u$ is as in the statement of the lemma. Then $\<u,g\>_{L^2(\sigma)}=0$, hence $v:=u\vert_{\partial \Omega}$ belongs to $\mathcal{H}$. This means that there are discrete signed measures $\tau_n$ supported on finitely many points in $\partial \Omega$  such that $U_{\tau_n}\to v$ in $L^2(\sigma)$. 
Replacing point masses by probability densities on $\partial \Omega$ concentrated near the points, we modify $\tau_n$ to obtain $\theta_n$ that are absolutely continuous with respect to $\sigma$ and such that $U_{\theta_n}\to v$ in $L^2(\sigma)$.  Hence,  $U_{\theta_n}$ are continuous in $\overline{\Omega}$.  As $U_{\theta_n}-u$ is harmonic in $\Omega$, this shows (e.g., by writing the Poisson integral formula) that $U_{\theta_n}\to u$ uniformly on compact subsets of $\Omega$. 
\end{proof}

While Theorem \ref{thm:areageneralcapacity1set} only proves that $\inf\kappa_n(K, 1)= 0$ for smooth enough compact sets $K$ with capacity $1$, it is nevertheless possible to get quantitative estimates for certain special class of compacts $K,$ namely closed unit lemniscates. We collect this as a Proposition below.

\bdefn\label{genK}
Let $K$ be a compact set of capacity $1$. We say that $K$ is \emph{generated by} a polynomial $q$, if $q$ is monic and 

\[K = \{z:|q(z)|\leq 1\}.\]

\noindent Since $q^j$ also generates $K$ for every $j\in\N,$ we assume that $d=deg(q)$ is the minimum degree among all monic polynomials which have $K$ as a unit lemniscate.
\edefn

\begin{proposition}\label{lemniscate bound}
 Let $K$ be a compact set of capacity $1$ generated by a polynomial $q$ with $\deg(q) = d$. Let $t\geq 1$ be fixed. Then, there exists $C = C(t)>0$ such that for all large enough $n$, we have
\[\kappa_n(K, t)\leq \dfrac{C}{(\log\log n)^{\frac{1}{d}}}\]  \end{proposition}

\begin{proof}
Let $\{p_n\}_n$ be the sequence of polynomials constructed in Step-4 of the proof of the upper bounds in Theorems \ref{thm:arealevel1} and \ref{thm:arealevelmorethan1} in Section \ref{sec:UB}. For each $n\in\N,$ define polynomials $f_n(z) = p_n(q(z)).$ Let us observe that the $f_n$ are monic polynomials with $\deg(f_n)= nd$. Note that the zeros of $f_n$ are the $q-$preimages of the zeros of $p_n$. Therefore by the definition of $K,$ it follows that all the zeros of $f_n$ are in $K.$ Next we easily see that the corresponding lemniscates are related as follows
\begin{equation}\label{polynlem}
 \Lambda_{f_n}(t) = q^{-1}(\Lambda_{p_n}(t))   
\end{equation}
We now use the following result of Crane \cite{Crane}.
\bthm\label{Crane}(Crane)
Let $q$ be a monic polynomial of degree $d.$ Let $U$ be a Lebesgue measurable set in the complex plane $\C.$ Then
\[ m(q^{-1}(U))\leq\pi \left(\frac{m(U)}{\pi} \right)^{\frac{1}{d}},\]
where equality holds iff $U$ is a disc, and $q$ has a unique critical value at the center of $U.$
\ethm

\noindent From Theorem \ref{thm:arealevelmorethan1}, we know that there exists $C_1>0$ such that for large enough $n,$ we have the estimate $m(\Lambda_{p_n}(t))\leq\frac{C_1}{\log\log n}$. Combining this with the relation \eqref{polynlem}, and applying Theorem \ref{Crane} (with $U = \Lambda_{p_n}(t)$) yields
\begin{equation}\label{cranebd}
m\left(\Lambda_{f_n}(t)\right)\leq\frac{C_2}{(\log\log n)^{\frac{1}{d}}} \leq\frac{C_3}{(\log\log (nd))^{\frac{1}{d}}}
\end{equation}
This proves the upper bound for $\kappa_j(K, t)$ for degrees $j$ belonging to $\mathcal{N}: = \{nd: n\in\N\}$. To estimate area bounds for other degrees, suppose that $l = nd + r$ for some $n\in\N$, and $0< r < d.$ Fix $z_0\in\partial K$. Consider $Q_l(z) = f_n(z)(z-z_0)^r.$ Then $\deg(Q_l) = l.$ Also observe that
\[\{|Q_l| < t\}\subset \{|f_n|< (\log n)^{\alpha}\}\cup B\left(z_0, \frac{t^{\frac{1}{d}}}{(\log n)^{\frac{\alpha}{d}}}\right)\]
From here, we use subadditivity of areas along with the estimate \eqref{cranebd} (which holds for $t=(\log n)^{\alpha}$) to obtain the desired upper bound on $m(\{\vert Q_l\vert < t\}$.
\end{proof}

\section{Proofs of lower bounds in Theorems~\ref{thm:arealevel1}, \ref{thm:arealevelmorethan1}, \ref{thm:arealevellessthan1}}\label{alllowerbounds}


In this section, we prove the lower bounds for $\kappa_n(\T,t)$ stated in Theorems \ref{thm:arealevel1}, \ref{thm:arealevelmorethan1}, and \ref{thm:arealevellessthan1}.  We also establish the lower bound for $\kappa_n(\overline{\D},t)$ stated in Theorem \ref{thm:arealevellessthan1} addressing levels $0<t<1$.
But the (comparison of constraints) estimate 
\ben
\frac{1}{3}\kappa_{n(\log n)^4}(\T,t) \leq \kappa_n(\overline{\D},t)
\een
stated in Theorems \ref{thm:arealevel1} and \ref{thm:arealevelmorethan1}
relies on different methods, and hence we defer its proof until Section \ref{sec:key} below.

One of the key tools in establishing the lower bounds addressed in this section is the theorem of Nazarov-Polterovitch-Sodin from \cite{NPS} stated above as Theorem \ref{NPS2}. 



\begin{remark}\label{NPSpolynom}
We will often use the Theorem \ref{NPS2} in the following context: Let $p$ be a polynomial of degree $n$. Then for any $t>0,$ with $h(z) = \log(\frac{|p(a + rz)|}{t})$, we have $\nu(\T, h)\le 2n$. To see this, let $M$ denote a Mobius transform taking the real line $\R$ to the circle $a + r\T$ . Denote $R(w) = p(M(w)).$ Note that $R$ is a rational function of degree $n.$ The number of sign changes of $h$ on $a+r\T$ is at most the number solutions to the system 
\begin{equation}\label{R1}
R(w)\overline{R(w)} = 1, \hspace{0.05in}\mbox{and}\hspace{0.05in} \{\Im (w) = 0\},   \end{equation}
which is at most $2n$ by the Fundamental theorem of Algebra. In addition, if $p$ has no zeros in the disc $a+ r\D,$ then $h$ is harmonic in said disc and Theorem \ref{NPS2} along with \eqref{R1} shows that $\{h<0\}\cap (a +r\D)$ has area at least $\frac{c}{\log n}$.
\end{remark}

\subsection{Level-1 lemniscates}
Let $p\in\mathcal{P}_n(\T)$. We wish to get a lower bound on $m(\Lambda_p)$. The function $u(z) = \log|p(z)|$ is harmonic in $\D$, with $u(0) = 0$. Note that $\{u < 0\} = \{|p| < 1\}.$ Consider the function $v(z) = u(0.98 z)$. This is harmonic in $\D$, continuous upto $\T$ and $v(0) = 0$. To obtain the lower bound, it is sufficient to estimate the area of the set $\{v< 0\}\cap\D$. With a view to apply Theorem \ref{NPS2} to $v$, we will estimate the corresponding doubling exponent. We will do this by proving an upper bound on $\sup_{\D}|v|$ and a lower bound on $\sup_{\frac{1}{2}\D}|v|.$ Towards this, first note that since the zeros of $p$ are all on $\T$, we have $|p(z)|\leq 2^n$ for $z\in\T$. Since $v(0) = 0,$ this gives 
\begin{align}\label{upperboundD}
\sup_{\D}|v|\leq n\log 2
\end{align}

Next observe that $\sup_{0.49\D}|p(z)|\geq {\vert\vert p\vert\vert}_{L^2(0.49\T)}$, where ${\vert\vert f\vert\vert}_{L^2(r\T)}^2 = \int_{0}^{2\pi}|f(re^{i\theta})|^2\frac{d\theta}{2\pi}$. It is easily checked that if $p(z) = \sum_{k=0}^{n}a_k z^k,$ we have ${\vert\vert p\vert\vert}_{L^2(r\T)}^2 = \sum_{k=0}^{n}|a_k|^2r^{2k}$. In our case, since $p$ is monic, we have $|a_n| = 1$. On the other hand, since all the zeros of $p$ are on $\T,$ we obtain that $|a_0| = 1.$ Putting these facts together and plugging $r= 0.49,$ we obtain that 
\[\sup_{0.49\D}|p|\geq{\vert\vert p\vert\vert}_{L^2(0.49\T)}\geq 1 + \exp(-cn),\] for some $c> 0.$ This yields that

\begin{align}\label{lowerbd0.5D}
\sup_{\frac{1}{2}\D}|v| \geq\exp(-c_1n)   
\end{align}

Plugging in the estimates \eqref{upperboundD} and \eqref{lowerbd0.5D} into the expression for $\beta(\D, v)$ we arrive at 

\[\beta(\D, v)\leq \log\left(\dfrac{n\log 2}{\exp(-c_1n)} \right)\leq c_2 n, \]

from which it follows that $\log\beta^{*}(\D, v)\leq c_3\log n.$ This completes the proof.

\subsection{Level above 1 lemniscates}

Let $t\in (1, \infty)$ be fixed. It is enough to prove the area lower bound for a sequence $p_n\in\mathcal{P}_n(\T)$ which satisfies $m(\Lambda_{p_n}(t))\to 0$ as $n\to\infty.$ Fix such a sequence. Let $u_n (z) = \log|p_n(z)|.$ Then $u_n$ is harmonic in $\D$ with $u_n(0) = 0.$ Now we have

\begin{equation}\label{doublingexponent2}
m(|p_n| < t )\geq m(\{|p_n| < 1\}\cap \D )= m(\{u_n < 0\}\cap\D)\geq \frac{c}{\log \beta^{*}(\D, u_n)},
\end{equation}

\noindent where the last line once again follows from an application of Theorem \ref{NPS2}. We next estimate $\beta(\D, u_n) = \log\left(\frac{\sup_{\D} |u_n|}{\sup_{\frac{1}{2}\D}|u_n|} \right)$ from above. Since $p_n$ has all its zeros on $\T,$ it follows as before, that $\sup_{\D}|u_n|\leq n \log 2.$ Also note that for all large enough $n,$ we have $\sup_{\frac{1}{2}\D}|p_n|\geq t,$ for otherwise along a subsequence $n_j,$ $|p_{n_j}| < t$ everywhere in $\frac{1}{2}\D$, contradicting that $m(\{|p_n| < t\})\to 0$ as $n\to\infty.$ This gives the lower bound $\sup_{\frac{1}{2}\D}|u_n|\geq \log(t).$ Plugging these upper and lower bounds into the expression for $\beta(\D, u_n),$ we obtain that for all large enough $n,$

\[\beta(\D, u_n)\leq \log\left(\frac{n\log 2}{\log t} \right)\leq c_1(t)\log n,\]
which in turn implies that $\log\left(\beta^{*}(\D, u_n)\right)\leq c(t)\log\log n$. Substituting this in the right most term of the estimate \eqref{doublingexponent2} proves the area lower bound for $t$-lemniscates and concludes the proof of the Theorem. 

\subsection{Level below 1 lemniscates}

In this section we prove the lower bound for $\kappa_n(\overline{\D}, t)$ as well as the improved lower bound for $\kappa_n(\T, t)$.

Let $t\in (0, 1)$ be fixed. We will first show that there exists a constant $c>0$ depending only on $t$ such that
\[\kappa_n(\overline{\D}, t)\geq\frac{c}{n^4}.\] 
The proof is essentially an argument of Pommerenke \cite{Pomm1961}, but for the reader's convenience we present it here.  Let $p_n\in\mathcal{P}_n(\overline{\D})$. Consider $\Lambda_n = \{z: |p_n(z)| < 1\}$, and denote by $U_n$ the component of $\Lambda_n$ that contains the origin. Pommerenke \cite[Theorem $3$]{Pomm1961} has shown that the diameter $d_n$ of $U_n$ satisfies $d_n > 1.$ Using $d_n > 1$, an application of Bernstein's inequality gives
\begin{equation}\label{Bernsteinpomm}
||p_n'||_{U_n}\leq \frac{2e}{d_n}n^2\leq 2en^2 
\end{equation}
Next, observe that some component $V_n$ of $\Lambda_n(t)$ is necessarily contained in $U_n.$ Let $w$ be a zero of $p_n$ contained in $V_n$ and $w^{*}\in\partial V_n$ be the point on the boundary closest to $w.$ Then integrating along the line segment $[w, w^{*}]$ we have
\begin{equation}\label{1/2inradius}
t= |p_n(w^{*}) - p_n(w)| = \left|\int_{w}^{w^{*}}p_n'(t)dt\right|\leq 2en^2|w - w^{*}|,  
\end{equation}
where we have used the estimate \eqref{Bernsteinpomm} to bound the integrand above. By our choice of $w^{*}$ it follows from \eqref{1/2inradius} that the ball $B(w, \frac{t}{2e n^2})\subset V_n\subset\Lambda_n(t)$. The area lower bound now follows.

We now prove the lower bound for $\kappa_n(\T, t)$ stated in Theorem \ref{thm:arealevellessthan1}, namely,
\ben\label{eq:TimprovedLB}
\kappa_n(\T,t)\ge \frac{C}{n^2\log n}.
\een
It seems likely that this can be further improved, see Proposition \ref{prop:evidence} below.

To prove the bound \ref{eq:TimprovedLB} we will need the following result.  As we shall see, it follows from (the proof of) a result of Wagner \cite{Wagner88}.

\begin{lemma}\label{lem:Wagner}
Let $g_n$ be a monic polynomial of degree $n,$ having all its zeros on $\T.$  
Fix $0<t<1$.
Let $\ell_1, \ell_2,..., \ell_k$ be the lengths of the component arcs of $\Lambda_{g_n}(t) \cap \T$.
Then, for some constant $\lambda> 0$ independent of $n$ (but depending on $t$), we have
\begin{equation}\label{diambd}
\sum_{j=1}^{k}\ell_j^2 \geq\frac{\lambda}{n}.
\end{equation}
\end{lemma}

\begin{proof}[Proof of Lemma]

Our starting point is the last displayed equation on page 332 in \cite{Wagner88}, which states that
\ben\label{eq:WagnerAux}
\sum_{j=1}^m \alpha_j \ell_j^* \geq \frac{1}{4n} \sum_{j=1}^m \alpha_j^2.
\een
Here $\ell_j^*$ denote the lengths of the component arcs associated to an auxiliary polynomial $g_n^{*}$ that has a single zero of multiplicity $\alpha_j$ in each of the component arcs of $\Lambda_{g_n^{*}}(t)\cap\T$, and such that $\Lambda_{g_n^{*}}(t)\cap\T\subset \Lambda_{g_n}(t)\cap\T$.
We note that the latter condition implies
\ben\label{eq:arcs}
\sum_{j=1}^k \ell_j^2 \geq \sum_{j=1}^m (\ell_j^*)^2.
\een

Squaring both sides of \eqref{eq:WagnerAux} and applying the Cauchy-Schwarz inequality to the expression on the left hand side, we obtain 
$$
\sum_{j=1}^m \alpha_j^2 \sum_{j=1}^m (\ell_j^*)^2 \geq \frac{1}{16n^2} \left( \sum_{j=1}^m \alpha_j^2\right)^2.
$$
Simplifying and then using $\sum_{j=1}^m \alpha_j^2 \geq n$ gives
\ba
\sum_{j=1}^m (\ell_j^*)^2 &\geq \frac{1}{16n^2}  \sum_{j=1}^m \alpha_j^2 \\
&\geq \frac{1}{16n}.
\ea
In light of \eqref{eq:arcs} this implies \eqref{diambd} and completes the proof of the lemma.
\end{proof}

Now we proceed to the proof of \eqref{eq:TimprovedLB}.
Let $p_n$ be any polynomial in $\mathcal{P}_n(\T)$.  Let $T_1, T_2,..., T_k$, be the components of $\Lambda_{p_n}(t) \cap \T$ with arc lengths $\ell_1, \ell_2, ..., \ell_k$ respectively. Note that $k\leq n,$ since the lemniscate of a degree-$n$ polynomial can intersect the circle in at most $2n$ points counted with multiplicity.

For $j\in[1, k],$ let $S_j$ be the sector such that $T_j = S_j\cap\T$.
For each $j$, we place $2n$ disjoint discs each centered at points on $\partial\Lambda_{p_ n}(t)\cap S_j$ as follows. First partition $T_j$ into $2n$ equal subarcs and take the central point in each subarc.  Then from each of these points take the ray directed away from the origin and select the  point with maximal modulus in the intersection of that ray with the closure of $\partial\Lambda_{p_n}(t)$.  Since moving points in these outward radial directions can only increase the distance between them, the new collection of points are mutually separated by at least $\frac{\ell_j}{4n}$, and we can choose disjoint discs of radius $\frac{\ell_j}{8n}$ centered at each point.  Since $p_n$ has only $n$ zeros, at least $n$ of the $2n$ discs are free of zeros.  We select these zero-free discs.
Applying Theorem \ref{NPS2} (with $h = \log\vert \frac{p_n}{t}\vert$), we have that the area of the intersection of $\{|p_n|<t\} $ with one of these discs is at least $\frac{c}{\log n} \pi \left( \frac{\ell_j}{8n}\right)^2 = \frac{\tilde{C}}{n^2 \log n} (\ell_j)^2$, and summing these estimates, first over the $n$ discs and then across all $j=1,2,...,k$, we have 
\[m\left(\Lambda_{p_n}(t)\right)\geq\frac{\tilde{C}}{n \log n}\sum_{j=1}^{k}\ell_j^2.\]
Combining this with Lemma \ref{lem:Wagner} immediately gives the desired lower bound.  This concludes the proof of \eqref{eq:TimprovedLB}.

For $t\in (0, 1)$ we believe that the sharp lower bound for area of $\Lambda_{p_n}(t)$ is of order $\frac{1}{n}.$ We give some evidence to support this conjecture by considering two extreme cases.

\begin{proposition}\label{prop:evidence}
Fix $t \in (0,1)$ and consider a sequence $p_n\in\mathcal{P}_n(\overline{\D})$. 
\begin{enumerate}[(a)]
\item If $\Lambda_{p_n}(t)$ is connected, then $m\left(\Lambda_{p_n}(t)\right)\geq\frac{c(t)}{n\log n}$.

\item If $\Lambda_{p_n}(t)$ has $n$ components, then $m\left(\Lambda_{p_n}(t)\right)\geq\frac{c(t)}{n}.$
\end{enumerate}
  
\end{proposition}

\begin{proof}[Proof of $(a)$]
We may assume that $m(\Lambda_{p_n}(t))\to 0$ as $n\to\infty$. In this case, note that by equidistribution of the zeros, c.f. Theorem \ref{zero distribution}, the diameter of $\Lambda_{p_n}(t)$ is at least $1.$ Now as in the above proof we construct $2n$ disjoint discs, each of radius $\frac{1}{4n}$ with centers on $\partial\Lambda_{p_n}(t)$. There are at least $n$ of these discs which are zero-free. Applying Theorem \ref{NPS2} as before, we obtain a lower bound of 
\[m\left(\Lambda_{p_n}(t)\right)\geq n \frac{c(t)}{n^2 \log n}= \frac{c(t)}{n \log n}. \]
This concludes the proof of (a).
\end{proof}

\begin{proof}[Proof of $(b)$.] Suppose that $\Lambda_{p_n}(t)$ has $n$ components. We will use the following two facts.

\noindent \textbf{Fact 1:} $\prod_{j=1}^{n}|p_n'(X_j)|\leq n^n$. Proof: Expand out the LHS writing it as a Vandermonde determinant and use Hadamard's inequality. 

\vspace{0.1in}

\noindent Let $U_j$, $1\leq j\leq n,$ be the components of $\Lambda_{p_n}(t)$.  Then $p_n$ restricted to each $U_j$ is a conformal map. Let $r_j = dist (X_j, \partial U_j).$ Then

\vspace{0.1in}

\noindent \textbf{Fact 2:} $\prod_{j=1}^{n}r_j\geq\frac{t^n}{(4n)^n}$

\noindent Proof of Fact $2$: Since $p_n$ is conformal on $U_j$ and maps $U_j$ onto $t\D,$ we have by Koebe's $\frac{1}{4}$ distortion inequality that $\frac{t}{4r_j}\leq |p_n'(X_j)|$. Combining these estimates for each $j$ along with the bound in Fact $1$, we may write

\begin{equation}\label{GM}
\prod_{j=1}^{n}r_j\geq\frac{t^n}{(4n)^n}.    
\end{equation}

\noindent With these facts, it is easy to finish the proof. Applying the AM-GM inequality to the $\{r_j^2\}$ and using Fact 2, we obtain

\[m(\Lambda_{p_n}(t))\geq \pi\sum_{j=1}^{n}r_j^2\geq \pi n\left(\prod_{j=1}^{n}r_j^2\right)^{\frac{1}{n}}\geq \pi n\frac{t^2}{16 n^2}= \frac{\pi t^2}{16n}.\]
This concludes the proof.
\end{proof}

\subsection{Possible strategies towards a $\frac{1}{n}$ lower bound}

Here we list two approaches which may lead to sharp bounds of order $\frac{1}{n}$ for $\frac{1}{2}-$level minimizers.
\begin{enumerate}
\item Define a zero $X$ to be \emph{M- good} if

\[\sup_{B\left(X, \frac{1}{100n}\right)} |f_n'(z)|\leq M n.\]   

\noindent If $X$ is \emph{M- good}, then notice that by the Bernstein bound we have $B\left(X, \frac{1}{200n}\right)\subset \Lambda_{f_n}(\frac{1}{2}).$  If we can show that there exist $M, \alpha> 0$ such that at least $\alpha n$ of the zeros are \emph{M- good}, then we would obtain the desired lower bound.  

\item A second strategy is to show that there for each $n\in\N$, there exists a minimizer $f_n$, all of whose critical values are greater or equal to $ \frac{1}{2}$, perhaps using an approach related to Eremenko and Hayman \cite[proof of Lemma 5]{EremHay}.
\end{enumerate} 

\subsection{Examples}

We list here some examples of polynomials in $\mathcal{P}_n(\T)$ all of which have completely disconnected lemniscates, thus having an area lower bound of order $\frac{1}{n}.$

\begin{enumerate}
\item Consider $p_n(z) = z^n - 1$.Then $p_n$ are uniformly bounded on $\T.$ The zeros of $p_n$ are the $n$th roots of unity $\omega_j$. Note that $p_n$ has a critical point of multiplicity $n-1$ at $z=0.$ All the critical values of $p_n$ are $1$ which is larger than the level $\frac{1}{2}$ in absolute value. It follows that $\Lambda_{p_n}(\frac{1}{2})$ has $n$ connected components.

\vspace{0.1in}

\item A related example is the sequence $q_n(z) = \frac{z^n - 1}{z - 1}.$ The zeros of $q_n$ are the $n$th roots of unity $\omega_j$, except $1.$ It can be shown by direct polynomial algebra that all the critical values of $q_n$ are strictly bigger than $\frac{1}{2}$ in absolute value. This forces once again $\Lambda_{q_n}(\frac{1}{2})$ to have $n$ components. Two features of this example to keep in mind: In contrast to Example $1$, the critical points of $q_n$ equidistribute on $\T,$ and second $||q_n||_{\T} = n.$ So while uniform boundedness on the whole disc will certainly yield at least order $\frac{1}{n}$ area, it is not necessary.

\vspace{0.1in}

\item Our last example is a further modification of Example $2.$ Let $t_n(z) = q_n(z^{\log n}).$ Then $t_n$ is monic, has degree $n\log n$. The zeros of $t_n$ are the preimages of $\{\omega_j\}_{j=1}^{n-1}$ under the map $z^{\log n}.$ Differentiating, we get
\[t_n' (z) = q_n' (z^{\log n})(\log n) z^{\log n -1}\]
From here it is easy to check that all the critical values of $t_n$ are bigger than $\frac{1}{2}.$ The feature of $t_n$ we want to keep in mind is that $||t_n|| = n,$ and simultaneously, $t_n$ is unbounded on every arc of $\T$ which has positive length. 

\end{enumerate}

\subsection{Close to $1$-lemniscates}

We can use the area lower bound for levels $t>1$ to get closer to the area of $1$-lemniscates. Indeed, let $p_n, q_n \in\mathscr{P}_n(\overline{\D})$ be $1-$level minimizers and $2-$level minimizers respectively. Let $\alpha>0,$ and $l_n = e^{(\log n)^{\alpha}}$. Then, we have

\ba
m\left(|p_n| < 1 + \frac{\log 4}{l_n}\right)&= m\left( |p_n|^{l_n} < \left(1 + \frac{\log 4}{l_n}\right)^{l_n}\right)\\
&\geq m\left( |p_n|^{l_n} < 2\right) \\
&\geq m\left(|q_{nl_{n}}| < 2 \right)\\
&\geq \frac{c_1}{\log \log (nl_n)} \\
&\geq \frac{c}{\log \log n}.
\ea
Here the third displayed line follows because $p_n^{l_n}\in \PP_{nl_n}(\overline{\D})$, and by assumption $q_{nl_n}$ is a $2-$level minimizer for that class. The fourth line follows from using the lower bound from Theorem \ref{thm:arealevelmorethan1}.

\begin{remark}
Let  $\Lambda_n'=\{1<|p|<1+\frac{\log 4}{l_n}\}$, with $l_n$ as above.  If it could be shown that $|\Lambda_n'|\ll \frac{1}{\log \log n}$, then it would follow (from the estimate above) that $m(\Lambda_n)\ge m(\Lambda_n(2))-m(\Lambda_n') \gtrsim \frac{1}{\log \log n}$.  This would resolve the sharp order of decay for the minimal area problem at level $t=1$.


\end{remark}



\section{Equidistribution and the zero-pushing lemma}

Returning to Theorem \ref{thm:arealevel1}, so far we have provided the proofs of estimates for $\kappa_n(\T,t)$. In this section, we finish the proof by establishing the desired comparison of constraints
\ben\label{eq:key}
\frac{1}{3}\kappa_{n(\log n)^4}(\T,t) \leq \kappa_n(\overline{\D},t).
\een
The proof of this estimate (provided in Section \ref{sec:key} below) will use two main ingredients: the zero-pushing lemma (Lemma \ref{pushzeros} below) used to replace the minimizer for the disc case with a comparable polynomial having zeros on the circle (with nonfatal increase in the degree), and an estimate (Lemma \ref{area in vs out} below) for the portion of the lemniscate outside the disc in terms of the portion inside. The latter relies on equidistribution of zeros for the minimizer (Theorem \ref{zero distribution}), a property of independent interest that we now establish.

\subsection{Equidistribution: proof of Theorem~\ref{zero distribution}}
Here we give two proofs that the zeros of the lemniscate-area-minimizing polynomials equidistribute on the unit circle. The proofs do not actually use the optimality of the polynomials, only that the lemniscate areas vanish in the limit.

The first proof is a compactness argument using potential theory, and also applies to other situations (e.g., equidistribution of zeros of polynomials minimizing the inradius). The second proof restricts to the case when the zeros are on the circle, but the end result is more quantitative.

\begin{proof} We assume to the contrary that some subsequence of the empirical measures $\mu_{n_j}\overset{w}{\to}\mu,$ where $\mu$ is a probability measure supported in $K$ but $\mu\neq\nu_K.$ Remembering that $cap(K) = 1$, we have by the energy minimizing property of $\nu_K$, c.f., \cite{ransford}, that 
\begin{equation}\label{energynegative}
\int_K U_{\mu}(z)d\mu(z)=\int_K\int_K \log|z-w|d\mu(w)d\mu(z) < 0.    
\end{equation}

Using \eqref{energynegative}, we conclude that there exists $z_0\in\supp(\mu)$ with $U_{\mu}(z_0) < 0.$ From here, using that $U_{\mu}$ is a subharmonic function, we deduce that there is a ball $B(z_0, r)$ in which $U_{\mu} < 0$. Subharmonic functions assume their maximum on any compact set, so for instance we have 
\begin{equation}\label{negativeinB}
\sup_{B(z_0, \frac{r}{2})}U_{\mu}\leq - c_1 < 0.    
\end{equation} To finish the proof we need the following lemma. 

\begin{lemma}\label{pottheory lemma1}
    Let $K\subset\C$ be compact. Let $\mu_n \in \mathcal{P}(K)$ be a sequence such that $\mu_n$ converges to $\mu$ in the weak* topology.
    Then for any compact set $F$ we have,
    \begin{align*}
       \underset{n \to \infty}{\overline{\lim}}  \sup_{F} U_{\mu_n} \leq \sup_{F} U_{\mu}.
    \end{align*}
   
\end{lemma}

This lemma is probably well known and is also proved in \cite{KoushikSubhajit}, but for completeness we prove it below. For a moment, assume the lemma holds. We show how to finish the proof of the Theorem. Indeed, applying Lemma \ref{pottheory lemma1} (with  $F= \overline {B(z_0, \frac{r}{2})}$) in conjunction with the estimate \eqref{negativeinB}, and remembering that $U_{\mu_n}(z) = \frac{1}{n}\log|p_{n}(z)|,$ we get that $|p_{n_j}|\leq\exp(-c_1n_j)$ everywhere on the ball $B(z_0, \frac{r}{2}).$ In other words,
\[B\left(z_0, \frac{r}{2}\right)\subset\{|p_{n_j}| < \exp(-c_1n_j)\}\subset\{|p_{n_j}| < t\},\]

\noindent for any fixed $t>0$, and all $j$ large enough. The above displayed inclusion shows that the area of the $t$-lemniscates along this subsequence is at least $\pi \frac{r^2}{4}$, and hence in particular do not approach zero. This contradicts the hypothesis in the statement and concludes the proof of the Theorem.
\end{proof}

We are now ready to prove Lemma \ref{pottheory lemma1}. 

\begin{proof}[\textbf{Proof of Lemma \ref{pottheory lemma1}}]
    Since $ U_{\mu_n}$ is a subharmonic function, its supremum on $F$ is attained at a point $z_n \in F.$ That is, 
    \begin{align}
        \sup_{F} U_{\mu_n} =  U_{\mu_n}(z_n), \quad  \forall n \in \N
    \end{align}
    Now suppose to the contrary that $\limsup_{n\to\infty}U_{\mu_n}(z_n) > \sup_{F}U_{\mu}.$ Then along a subsequence $n_j$ we will have
    \begin{equation}\label{descent}
    U_{\mu_{n_j}}(z_{n_j}) > \sup_{F}U_{\mu} +\eps.   
    \end{equation} 
    for some $\eps > 0.$ Since $z_{n_j}$ all lie in the compact set $F$,  we can get a further subsequence (which for notational ease, we assume to be just $n_j$), and a point $\tilde z \in F$ such that $ z_{n_j} \to \tilde z $. Next, we use the principle of descent (c.f. \cite[Theorem $6.8$]{SaffTotikLogpotential}) which states that if a sequence of probability measures $\{\mu_n\}_n$, all having support in a fixed compact set, converge to a measure $\mu$ in the $\mbox{weak}^{*}$ topology, then  we have 
    \begin{align*}
        \underset{n \to \infty }{\overline{\lim}} U_{\mu_{n}}(w_{n}) \leq  U_{\mu}(w^{*}),
    \end{align*}
     whenever $w_n\in\C$ with $w_n\to w^{*}.$ Applying the aforementioned result with $w_n=z_n$, and $w^* = \tilde z,$ we obtain $\limsup_{j\to\infty}U_{\mu_{n_j}}(z_{n_j})\leq U_{\mu}(\tilde z).$ Combining this with the estimate \eqref{descent} we arrive at 
     
     \[\sup_{F}U_{\mu} +\eps\leq \limsup_{j\to\infty}U_{\mu_{n_j}}(z_{n_j})\leq U_{\mu}(\tilde z).\]
\noindent Since $\tilde z\in F,$ this is absurd. Hence the proof of the lemma is complete. 
\end{proof}

Now we give an alternate proof of the equidistribution of zeros of area minimizers when $K=\overline{\D}$, but now it is assumed that the zeros are on the unit circle. The advantage of this proof is that it gives quantitative estimate of the discrepancy.

\begin{proof}[Proof of Equidistribution assuming zeros on the unit circle]
Let $f(z)=(z-w_1)\ldots (z-w_n)$ where $|w_i|=1$ for each $i$. Let $u=\frac{1}{n}\log |f|$ and without loss of generality, let $\mu=\max\limits_{\overline{\D}}u=u(1)$. For $0<r<1$, we have
\ba
\frac{|f(r)|}{e^{n\mu}}= \prod_{k=1}^n\frac{|r-w_k|}{|1-w_k|} \ge \l(\frac{1+r}{2}\r)^n
\ea
as $|r-w|/|1-w|$ is minimized when $w=-1$. Thus $u(r)\ge \mu+\log\frac{1+r}{2}$. We shall choose $r=r_{\mu}\vee \frac12$, where $\log\frac{1+r_{\mu}}{2}=-\frac{\mu}{2}$. In particular, we always have $u(r)\ge \frac{\mu}{2}>0$.

As $u$ is harmonic on $\D$, we have the standard estimate\footnote{In general, if $u$ is harmonic on $\Omega\subseteq \R^k$, then $\|\nabla u(x)\|\le \frac{k}{d(x,\partial \Omega)}\|u\|_{\Omega}$.}  $\|\nabla u\|\le \frac{2\mu}{1-r}$ on $\D(0,r)$. Therefore, $u\ge \frac12 u(r)$ on $\D(r,\delta)\cap \D(0,r)$, where $\delta=\frac{(1-r)u(r)}{4\mu}$. Consequently (as $u(r)>0$), we have 
\ba
\int_{r\D}u_+ \ge \frac{u(r)}{2}\times \frac{\pi}{4}\delta^2\ge  \frac{(1-r)^2u^3(r)}{64\mu^2}.
\ea
Here we used the fact that $\delta<\frac{1-r}{4}<r$ (as $r\ge \frac12$), hence at least a quarter of $\D(r,\delta)$ lies within $\D(0,r)$.

Since $u$ is harmonic in $\D$ and $u(0) = 0$, its integral on $r\D$ must be zero, hence $\int_{r\D} u_-=\int_{r\D} u_+$. Writing $b=m(\{u<0\}\cap r\D)$ and observing that $u\ge \log(1-r)$ everywhere on $r\D$ (because $|z-w_k|\ge 1-r$ for any $k$ and any $z\in r\D$), we see that 
\ba
b\log\frac{1}{1-r}\ge  \frac{(1-r)^2u^3(r)}{64\mu^2}.
\ea

\noindent\textbf {Case $1$}: $r=r_{\mu}$.  Then $u(r)\ge \frac12 \mu$, and $1-r=2(1-e^{-\frac{\mu}{2}})\in [ \frac{\mu}{4},\frac{\mu}{2}]$. Therefore, $b\ge \frac{\mu^3}{2^{13}\log(4/\mu)}$, which implies that $\mu\lesssim \sqrt[3]{b\log\frac{1}{b}}$.

\noindent\textbf {Case $2$}: $r=\frac12$. Then we still have $u(r)\ge \frac12 \mu$ (as $r\ge r_{\mu}$) and hence $b\ge \frac{\mu}{2^{8}\log 2}$,  which implies that $\mu\lesssim b$.

\medskip
Thus in all cases, $\mu\lesssim \sqrt[3]{b\log \frac{1}{b}}$.

\medskip
\noindent Now we apply the well known Erd\"{o}s-Turan inequality (recall that $q$ is monic and $q(0)=1$) to obtain
\ba
\sup_{0\le \alpha<\beta\le 2\pi}\Big|\frac{1}{n}|\{k\suchthat \alpha <\arg w_k<\beta\}| - \frac{\bet-\alp}{2\pi} \Big|\lesssim \sqrt{\mu} \lesssim \sqrt[6]{b\log\frac{1}{b}}. 
\ea
In particular, if $p_n$ is a sequence for which $m(\Lambda_{p_n})\to 0$, then the zeros of $p_n$ are asymptotically equidistributed on the  unit circle.

\smallskip
\noindent{\textbf Remark}: If we instead use the fact that circle averages of $u$ are zero, for circles $s\mathbb T$ with $r-\frac34 \delta<s<r-\frac14 \delta$ to get a lower bound for the integral of $u_-$ on $s\mathbb T$ and then integrate over $s$, we seem to get the slightly better bound of $\sqrt[5]{b\log\frac{1}{b}}$.
\end{proof}

\subsection{The zero-pushing lemma}

As mentioned in the overview in Section \ref{sec:overview}, we state two slightly different versions of the zero-pushing lemma (Lemmas \ref{pushzeros} and \ref{lem:probabilistic} below) based on two methods of proof, with the proof of Lemma \ref{lem:probabilistic} being probabilistic.
We discuss the relative merits of the two proofs in Remark \ref{rmk:merits} below.  

\begin{lemma}[{Zero-pushing lemma}]\label{pushzeros}
Let $p\in\PP_n(\overline{\D})$ have at least one zero in $\D.$ Let $\eps > 0,$ and $L = [{\frac{6}{\eps^2}}].$ There exists a monic polynomial $q$ of degree $L n$ with all its zeros on the unit circle $\T$, and which satisfies
\begin{equation}\label{smallArea}
\log|p(z)|\leq \frac{1}{L}\log|q(z)|, \hspace{0.05in} \forall z\in (1-\eps)\D.
\end{equation}
\end{lemma}

\bprf [Proof of Lemma \ref{pushzeros}] 
Let $\eps>0$ be given. Let $p(z) = \prod_{j=1}^{n}(z - z_j)$ be in the collection $\mathcal{P}_n(\overline{\D})$ with at least one $z_j\in\D.$ Let
\[Z_{\eps}^1 = \{z_j: |z_j|\leq (1-\eps)\}, \hspace{0.05in} Z_{\eps}^2=\{z_j: |z_j|> (1-\eps)\}.\]

\noindent Denote $a_n = |Z_{\eps}^1|.$ We split $p = p_1p_2,$ where $p_j$ are monic polynomials with zero set $Z_{\eps}^j$ respectively, for j=1, 2. Note that  with this $\deg (p_1) = a_n$ and $\deg(p_2) = n-a_n.$

\vspace{0.1in}

\noindent Let $\alpha$ be a fixed zero of $p_2$. Since $|\alpha| > 1-\eps,$ it readily follows that
\[\log |z- \alpha|\leq \log \left|z - \frac{\alpha}{|\alpha|}\right|, \hspace{0.05in} z\in (1-\eps)\D. \]
Thus pushing each zero of $p_2$ radially to $\T,$ we get a monic polynomial $q_2$ of degree $n-a_n$ with all its zeros in $\T,$ such that
\begin{equation}\label{p_2}
\log |p_2(z)|\leq \log|q_2(z)| \hspace{0.05in}, \forall z\in (1-\eps)\D.    
\end{equation}

Next we deal with the zeros of $p_1.$ Let $|w|\le 1-\eps$ be a fixed zero of $p_1.$ Our goal is to find $L$ points $\xi_1,\ldots ,\xi_L$ on $\T$ so that $\frac{1}{L}\sum_{k=1}^L\log|z-\xi_k|\ge \log |z-w|$ whenever $|z|\le 1-\eps$. The points $\{\xi_k\}$ are allowed to depend on $w$, but the number $L$ should be the same for all $w\in Z_{\eps}^1.$

If $w=0$, we may take $\zeta_j=e^{2\pi i j/L}$, $1\le j\le L$. Then
\ba
\frac{1}{L}\sum_{k=1}^L\log|z-\zeta_k|- \log |z|&= \frac{1}{L}\log\frac{|z^L-1|}{|z|^L}
\ea
which is positive if $|z|^L\le \frac12$ for all $z\in (1-\eps)\D$, or equivalently if $(1-\eps)^L\le \frac12$. As $1-\eps< e^{-\eps}$, this   is satisfied if $L\ge \frac{\log 2}{\eps}$.

For general $w$, we take $\xi_k=B_{-w}(\zeta_k)$, where $B_a(z)=\frac{z-a}{1-\bar{a}z}$ and $\zeta_k$ are as above. Then,  
\ba
\frac{1}{L}\sum_{k=1}^L\log|z-\xi_k|&=\frac{1}{L}\sum_{k=1}^L\log \big|z-\frac{\zeta_k+w}{1+\bar{w}\zeta_k}\big|\\
&= \frac{1}{L}\sum_{k=1}^L\log \big|z-w-\zeta_k(1-\bar{w}z)|-\frac{1}{L}\sum_{k=1}^L\log|1+\bar{w}\zeta_k| \\
&= \log|1-\bar{w}z|+\frac{1}{L}\sum_{k=1}^L \log |B_wz-\zeta_k|-\frac{1}{L}\sum_{k=1}^L\log|1+\bar{w}\zeta_k| \\
&=  \log|1-\bar{w}z|+\frac{1}{L}\log|(B_wz)^L-1| -\frac{1}{L}\log|1-(-\bar{w})^L|
\ea
where we have used the fact that $\zeta_k$ are the $L$th roots of unity. Therefore,
\ba
\frac{1}{L}\sum_{k=1}^L\log|z-\xi_k| - \log|z-w|&=\frac{1}{L}\log \frac{|(B_wz)^L-1|}{|1-(-\bar{w})^L|\times |B_wz|^L}.
\ea
Let $r_{\eps}=\max\{|B_wz|\suchthat z,w\in (1-\eps) \D\}$. The maximum is attained when $w=1-\eps$ and $z=-w$, leading to 
\ba
r_{\eps}=\left|\frac{-2w}{1+w^2}\right|=1-\frac{(1-w)^2}{1+w^2}\le 1-\frac12 \eps^2\le  e^{-\frac12\eps^2}.
\ea
Consequently, 
\ba
\frac{|(B_wz)^L-1|}{|1-(-\bar{w})^L|\times |B_wz|^L} &\ge \frac{1-r_{\eps}^L}{(1+(1-\eps)^L)r_{\eps}^L} \; \ge \;  \frac{1-e^{-\frac12 L\eps^2}}{2e^{-\frac12 L \eps^2}}.
\ea
This is at least $1$ if $3e^{-\frac12L\eps^2}\le 1$ or equivalently if $L\ge \frac{2\log 3}{\eps^2}$. 

\vspace{0.1in}

Applying the above procedure to each zero of $p_1$, we obtain a monic polynomial $q_1$ of degree $La_n$ with all its zeros in $\T,$ such that
\begin{equation}\label{p_1}
\log |p_1(z)|\leq \frac{1}{L}\log |q_1(z)|, \hspace{0.05in} \forall z\in (1-\eps)\D.    
\end{equation}

From \eqref{p_1} and \eqref{p_2}, we conclude that for $z\in (1-\eps)\D,$
\ba
\log|p(z)|&= \log|p_1(z)| + \log |p_2(z)|\\
&\leq \frac{1}{L}\log|q_1(z)| + \log |q_2(z)|\\
&= \frac{1}{L} \log|q(z)|,
\ea
\noindent where $q(z) = q_1(z)q_2^L(z).$ Clearly $q$ is a monic polynomial with all its zeros in $\T.$ Furthermore $\deg (q) = La_n + (n-a_n)L = Ln.$ This finishes the proof of the lemma. 
\eprf

\begin{remark}\label{rmk:merits}
The above proof is a constructive ``derandomization'' of the original proof we came up with for the following alternative version of Lemma \ref{pushzeros}, which uses a probabilistic method for showing the existence of $p_1$, namely, ``pushing'' the zeros by replacing them with random samples (partly inspired by the random coefficient replacement used by K\"orner \cite{Korner} to solve a problem of Littlewood on the existence of so-called flat polynomials). 

While nonconstructive, the probabilistic approach relies on more general tools that could be adapted to settings where explicit formulas are unavailable.  Also, while the derandomized proof gives tighter control on the increase in degree (providing in general a smaller factor $L$), we note from the description of $L$ in the statement of Lemma \ref{lem:probabilistic} that with the probabilistic method we can preserve degree (i.e., we can take $L=1$, pushing each zero by replacing it with a single random sample) if $a_n$ is assumed to grow sufficiently fast.

The probabilistic approach is centered around having a certain point-wise estimate hold with overwhelming probability as $n \rightarrow \infty$, so the statement of the lemma specifies an appropriate choice of $\e$ and $L$ in terms of $n$.
\end{remark}

\begin{lemma}\label{lem:probabilistic}
Let $p\in\PP_n(\overline{\D})$ have at least one zero in $\D,$ and let $a_n = \deg(p_1)$ with $p = p_1 p_2$ factored as in the proof of Lemma \ref{pushzeros}. 
Let $\eps = (\log n)^{-2},$ and $L = \lceil \max\{(\log n)^{10} / a_n, 1 \} \rceil.$ There exists a monic polynomial $q$ of degree $L n$ with all its zeros on the unit circle $\T$, and which satisfies
\begin{equation}\label{smallArea2}
\log|p(z)|\leq \frac{1}{L}\log|q(z)|, \hspace{0.05in} \forall z\in (1-\eps)\D.
\end{equation}
\end{lemma}

\begin{proof}[Proof (probabilistic approach)]

Using the same decomposition $p=p_1 p_2$ from the above proof, apply the same direct strategy for replacing $p_2$ with $q_2$. 
Then we use a probabilistic method to show the existence of a suitable replacement $q_1$ for $p_1$ satisfying \eqref{p_1}.

First, we note that for $w,z \in \D$ we have
\ben\label{eq:singlelayer}
\log|z-w| \leq \int_{0}^{2\pi} \log|z-e^{i \theta}| f_w(\theta) d\theta,
\een
where $f_w(\theta)$ denotes the Poisson kernel.
Indeed, we have $\log|z-w| \leq \log|1-\bar{w}z|$ for $z,w \in \D$, and the latter is harmonic in $\D$ so it coincides with its Poisson integral
\ba
\log|1-\bar{w}z| &= \int_{0}^{2\pi} \log|1-e^{-i \theta} z| f_w(\theta) d\theta \\
&= \int_{0}^{2\pi} \log|z- e^{i\theta}| f_w(\theta) d\theta.
\ea

For each $i=1,...,a_n$ take $L$ independent random samples $z_{i,j}^* \in \T$ according to the density $f_{z_i}(\theta)$.
Let $q_1(z) = \prod_{i=1}^{a_n}\prod_{j=1}^L(z-z_{i,j}^*)$.

It follows from \eqref{eq:singlelayer} that
\ben\label{eq:expectation}
\EE \log|z-z_{i,j}^*| = \log|1-\bar{z_i} z|.
\een
We also have, for $|w|<1-\e$ and $|z|=1-\e$
\ben\label{eq:wiggleroom}
\log|z-w| \leq \log|1-\bar{w}z| - \frac{\e^2}{2}.
\een
Together \eqref{eq:expectation} and \eqref{eq:wiggleroom} imply
\ben\label{eq:anchor}
\log|p_1(z)| \leq \frac{1}{L} \EE \log|q_1(z)| - a_n \frac{\e^2}{2}.
\een

The desired uniform estimate \eqref{p_1} follows from an epsilon-net argument once we establish that, with overwhelming probability (w.o.p.), the following pointwise estimate holds for fixed $z$ with $|z| = 1-\e$.
\ben
\log|p_1(z)| < \frac{1}{L} \log|q_1(z)| - \frac{a_n}{4} \e^2 \quad \text{(w.o.p.)}
\een
and in view of \eqref{eq:anchor} it suffices to show that for each $|z|=1-\e$ we have
\ben\label{eq:wop}
\left| \frac{1}{L} \log|q_1(z)| - \EE \frac{1}{L} \log|q_1(z)| \right| \leq a_n \e^2/8 \quad \text{(w.o.p.)}
\een
Since $|z_{i,j}^*|=1$ we have
\ben\label{eq:boundedRV}
\log \e \leq \log |z - z_{i,j}^*| \leq \log 2, \quad \text{for } |z|=1-\e.
\een
Define (independent) random variables $X_{i,j} = X_{i,j}(z) := \frac{\log|z-z_{i,j}^*|}{-\log \e}$, and let
$\displaystyle S_N = \sum_{i=1}^{a_n} \sum_{j=1}^L X_{i,j}$ where $N= L \cdot a_n$. By \eqref{eq:boundedRV} $X_{i,j}$ are bounded, namely $|X_{i,j}| \leq 1$, and hence applying Hoeffding's inequality we have
\ben
\bP (|S_N - \EE S_N | \geq \lambda ) \leq  \exp \left(- \frac{\lambda^2}{2N}\right).
\een
Taking $\lambda = \frac{L a_n \e^2}{-8\log \e}$ we have the following estimate for the probability of the ``bad'' event
\ben\label{eq:bad}
\bP \left( \left| \frac{1}{L}\log|q_1(z)| - \EE \frac{1}{L}\log|q_1(z)| \right| \geq \frac{a_n \e^2}{8} \right) \leq  \exp \left(- \frac{L a_n \e^4}{128(\log \e)^2} \right).
\een
Hence, we yield the desired overwhelming probability statement \eqref{eq:wop} if we choose, as stated in the lemma, $\e = (\log n)^{-2}$ and $L = \lceil \max\{(\log n)^{10} / a_n, 1 \} \rceil$.
\end{proof}



\subsection{Establishing \eqref{eq:key} using equidistribution and zero-pushing}\label{sec:key}

For a polynomial $p,$ let $\Lambda_p^{*} = \Lambda_p\cap\overline{\D}$, and $\widehat\Lambda_p = \Lambda_p\cap\D^c$ denote the parts of the lemniscate inside and, outside the unit disc respectively. For a general polynomial $p$ we may not be able to compare $m(\Lambda_p^{*})$ and $m(\Lambda_p).$ However, for $p$ a monic polynomial with all zeros in $\overline{\D}$, it was shown in \cite{EHP} that $m(\Lambda_p)\leq 4\sqrt{\pi} \sqrt{m(\Lambda_p ^{*})}.$ We now show that this estimate can be improved for polynomials whose zeros are equidistributed. 

\begin{lemma}\label{area in vs out}
Let $p_n\in\mathcal{P}_n(\overline{\D})$ be a sequence of polynomials with corresponding lemniscates $\Lambda_n$, with $\mu_n$ denoting the empirical measure of zeros. Assume that 
$\mu_n\overset{w}{\to}\mu_{\T}$.
Then, for all large enough $n$ we have
\[ m(\Lambda_n)\leq 3 m(\Lambda_n^{*}).\]
\end{lemma}

This constant $3$ can be improved, but we opt for keeping the presentation simple.


\begin{proof}
First note that $\mu_n\overset{w}{\to}\mu_{\T}$ implies  $U_{\mu_n}\to U_{\mu_{\T}}$ uniformly on compact subsets of $\D^c.$ Using $U_{\mu_{\T}}(z) = \log^{+}|z|,$ this translates to $|p_n(z)|$ being exponentially large on any compact subset of $\D^c.$ In particular, for all  large enough $n,$
\begin{equation}\label{outsidesmall}
 \Lambda_n\subset 1.2\D.   
\end{equation}

From here, our argument is a slight modification of that in \cite{EHP}. Following in their vein, we first observe that for our sequence of polynomials $p_n$ a simple computation (using the fact that all the zeros are in $\overline{\D}$) yields the estimate
\begin{equation}\label{reflectionestimate}
 |p_n(re^{i\theta})|\leq |p_n\left((2-r)e^{i\theta}\right)|, \quad r\in (0, 1).   
\end{equation}

Define the map $R:\D\rightarrow\C$ by $R(re^{i\theta}) = (2-r)e^{i\theta}.$ It is easy to check that that the Jacobian determinant of this map is given by $J(z) = \frac{2}{|z|} - 1.$ From \eqref{outsidesmall} and \eqref{reflectionestimate}, it follows that

\[\widehat{\Lambda_n}\subset R\left(\Lambda_n^{*}\cap\{z: 0.8 < |z| < 1\}\right):= R(U_n)\]

Therefore using a change of variables we have 
\ba\label{changeofvariable}
m(\widehat{\Lambda_n})&\leq m\left(R(U_n)\right) \\
&=\int_{R(U_n)}dm(z) \\
&=\int_{U_n}J(z)dm(z) = \int_{U_n}\left(\frac{2}{|z|} - 1\right) dm(z)\\
&\leq \frac{3}{2} m(U_n)\leq\frac{3}{2} m(\Lambda_n^{*}),
\ea
where in the last line we have used $0.8\leq |z|\leq 1$ for all points $z\in U_n$. The above estimate along with $m(\Lambda_n) = m(\widehat{\Lambda_n}) + m(\Lambda_n^{*}),$ finally gives $m(\Lambda_n)\leq \frac{5}{2}m(\Lambda_n^{*})$, thereby completing the proof.
\end{proof}

Having proved the above estimate, it is a simple matter now to establish \eqref{eq:key} showing that the minimum area for polynomials with zeros on $\overline{\D}$ and on $\T$ are essentially comparable (with marginal increase in the degree).

\begin{corollary}\label{minimum areas D vs T}
For all large enough $n,$  $\frac{1}{3}\kappa_{n(\log n)^4}(\T, 1)\leq\kappa_n(\overline{\D}, 1)$. On the other hand, if $t>1$, then for all large enough $n$, we have $\frac13 \kappa_{n(\log \log n)^4}(\T,t) \le \kappa_n(\overline{\D},t)$.
\end{corollary}

\begin{proof}
Define 
\begin{equation}\label{kappa*}
\kappa_n ^{*}(\D, 1) = \inf_{p\in\mathcal{P}_n(\overline\D)}\{m(\Lambda_p^{*})\}.    
\end{equation}
From the proof of the zero pushing lemma, Lemma \ref{pushzeros}, with $\eps = \frac{1}{(\log n)^2}$ we obtain

\begin{equation}\label{kappacompare}
\kappa_n (\overline{\D}, 1)\geq \kappa_n ^{*}(\D, 1)\geq \kappa_{n (\log n)^4}^{*}(\T, 1) - \frac{c}{(\log n)^2},  
\end{equation}
where $\kappa_n ^{*}(\T, 1)$ is defined as in \eqref{kappa*}, except that the infimum is taken over $\mathcal{P}_n(\T)$. Now we finish the proof by noting that from Lemma \ref{area in vs out} we have $\kappa_{n (\log n)^4}^{*}(\T, 1)\geq\frac{2}{5}\kappa_{n (\log n)^4}(\T, 1).$ For $t>1$, the proof of $\frac13 \kappa_{n(\log \log n)^4}(\T,t) \le \kappa_n(\overline{\D},t)$ follows in a similar vein with the obvious modifications. 
\end{proof}

\section{Estimate for inradius in terms of area}\label{sec:Solynin}

We start with a lemma relating the area, perimeter, and inradius of a sufficiently nice simply connected domain.

\begin{lemma}\label{lem1}
Let $\Omega\subset\C$ be a bounded simply connected domain with rectifiable boundary. Let $A, L$, and $\rho,$ denote the area of $\Omega,$  the length of $\partial\Omega$, and the inradius of $\Omega$ respectively. Then, 
\[A\leq 18\pi \rho L.\]
\end{lemma}

\begin{proof}
For $z\in\Omega,$ let $r_z = \dist(z, \partial\Omega)$. Then $\overline\Omega\subset\bigcup_{z\in\Omega}B(z, 2r_z).$ By compactness, a finite number of them cover $\Omega,$ say $\Omega\subset\cup_{i=1}^{m}B(z_i, 2r_i)$, where for convenience, we denote $r_{z_i}$ by $r_i$. Now we deal with two cases.

\vspace{0.1in}

\noindent \textbf{Case $1$}: $\Omega\subset B(z_i, 2r_i)$ for some $i,$ $1\leq i \leq m.$

\vspace{0.1in}

Then $A\leq 4\pi r_i^2,$ and since $B(z_i, r_i)\subset\Omega,$ we have $L \geq 2\pi r_i$. Putting these two together we have $A\leq 2 r_i L\leq 2\rho L$, and hence the conclusion of the lemma holds in this case.

\vspace{0.1in}

\noindent \textbf{Case $2$}: None of the $B(z_i, 2r_i)$ contains all of $\Omega$.

\vspace{0.1in}

In this case, one easily checks that the length of $\partial\Omega$ contained in these balls cannot be too small, namely 
\begin{equation}\label{annulus}
|\partial\Omega\cap T(z_i, r_i, 2r_i)| \geq 2r_i.  
\end{equation}
Here $T(z, r, s)$ denotes the annulus with center at $z,$ and inner and outer radii $r$ and $s$ respectively. Now we apply the Vitali covering lemma to obtain a collection of disjoint balls, $\{B(z_{i_j}, 2r_{i_j})\}_{j=1}^k$ such that
\[\Omega\subset\bigcup_{j=1}^{k}B(z_{i_j}, 6r_{i_j})\]
This implies that 
$A\leq 36\pi \sum_{j=1}^{k}r_{i_j}^2.$
On the other hand, using the fact that the discs $B(z_{i_j}, 2r_{i_j})$ are disjoint, along with the estimate \eqref{annulus}, we obtain
$L\geq 2 \sum_{j=1}^{k} r_{i_j}.$
Combining the above two estimates, we obtain
$$A\leq 36\pi (\sup_{1\leq j \leq k}r_{i_j})\sum_{j=1}^{k} r_{i_j}\leq 18\pi (\sup_{1\leq j \leq k}r_{i_j})L\leq 18\pi \rho L,$$
which proves the lemma.
\end{proof}

\begin{lemma}\label{lemma:FN}
As in the statement of Lemma \ref{lem:inradius}, let $t>0$ and let $p$ be a (not necessarily monic) polynomial of degree $n$. Let $\Lambda = \Lambda_p(t)$.
The area $A(\Lambda)$ of $\Lambda$ and its perimeter $L(\Lambda)$ satisfy
\begin{equation}\label{eq:FNconcl}
L(\Lambda) \leq 4n\sqrt{\pi} \sqrt{A(\Lambda)}.
\end{equation}
\end{lemma}

\begin{proof}[Proof of Lemma \ref{lemma:FN}]
We have \cite[Sec. 4]{FN}
\begin{equation}\label{eq:FN}
L(\Lambda) \leq \sum_{\xi} \int_{\Lambda} \frac{1}{|z-\xi|} dA(z),
\end{equation}
where the sum is over all $2n-1$ points $\xi$ satisfying $p(\xi)p'(\xi)=0$.
(In \cite{FN} it was assumed that $t=1$ and that $p$ is monic in the definition of $\Lambda$, but \eqref{eq:FN} holds for arbitrary $t>0$ and for $p$ not necessarily monic.)

The Schwarz rearrangement inequality gives
\begin{equation}
\int_{\Lambda} \frac{1}{|z-\xi|} dA(z) \leq \int_{D_\Lambda} \frac{1}{|z|} dA(z),
\end{equation}
where $D_\Lambda$ denotes the disc centered at the origin with the same area as $\Lambda$.
Let $r_\Lambda$ denote the radius of $D_\Lambda$.
Then 
\begin{equation}
\int_{D_\Lambda} \frac{1}{|z|} dA(z) = 2\pi r_\Lambda,
\end{equation}
and 
\begin{align*}
2\pi r_\Lambda &= 2\sqrt{\pi} \sqrt{\pi r_{\Lambda}^2} \\
&= 2\sqrt{\pi} \sqrt{A(\Lambda)},
\end{align*}
in view of the way $D_\Lambda$ was chosen.
This gives
\begin{equation}
L(\Lambda) \leq (2n-1)2\sqrt{\pi} \sqrt{A(\Lambda)},
\end{equation}
which implies \eqref{eq:FNconcl} and completes the proof of the lemma.
\end{proof}

\begin{proof}[Proof of Lemma \ref{lem:inradius}]
From Lemma \ref{lem1} we have
\begin{equation}\label{eq:fromLemma}
\rho(\Lambda) \geq \frac{1}{18\pi} \frac{A(\Lambda)}{L(\Lambda)}.
\end{equation}
The result now follows from Lemma \ref{lemma:FN} and \eqref{eq:fromLemma}.
\end{proof}

\begin{remark}
The estimate \eqref{eq:inradArea} is asymptotically sharp apart from the value of the coefficient $\frac{1}{72\pi \sqrt{\pi}}$.
Indeed, the area of the Erd\"os lemniscate $\{ |z^n-1| < 1 \}$ approaches a constant, while the inradius is asymptotically proportional to $n^{-1}$.
As stated in \cite{Solynin}, the Erd\"os lemniscate seems to be a likely candidate for an extremal case in this problem.
\end{remark}

\section{Normality of area minimizing polynomials}\label{normalsection}
In this section we analyze the normality of area minimizing polynomials for various levels. For $t>0,$ we recall that by a $t$-level minimizer we mean a polynomial $p_{n, t}\in\mathcal{P}_n(\overline{\D})$ which attains the infimum $\kappa_n(\overline{\D},  t)$. We refer to the overview section for the proof of non-normality of the class $\mathcal{F}(t)$ for $t\in (1, \infty)$. 

\vspace{0.1in}

\noindent We next show that minimizers for levels $t\in (0,1)$ form a normal family in $\D$. For simplicity we only prove the result for $t=\frac 12.$ Before we proceed, recall that for the sequence $p_n^*(z)= z^n - 1$, we have that   $m(|p_n^*|< \frac{1}{2})\leq \frac{c_0}{n}$ for some absolute constant $c_0>0.$ Hence any sequence of $\frac{1}{2}$ level minimizers $\{f_n\},$ satisfy the upper bound
\begin{equation}\label{aprioribound}
m\left(\Lambda_{f_n}\left(\frac12\right) \right)\leq \frac{c_0}{n}
\end{equation}

\bthm\label{Level 0.5}
 Let $\{f_n\}_n$ be a sequence of $\frac{1}{2}$-level minimizers. Then, $\{f_n\}$ form a normal family in $\D$.
 


\ethm

\begin{proof}
Let $f_n(z) = \prod_{j=1}^{n}(z - a_{j,n})$, where $\vert a_{j,n}\vert\leq 1$. Consider the associated sequence of polynomials defined by $g_n(z) = \prod_{j=1}^{n}(1-\overline{a_{j,n}}z)$. Note that $g_n$ has all its zeros on $\D^c$ and  
\begin{equation}\label{fdomg}
\vert f_n(z)\vert\leq\vert g_n(z)\vert, \hspace{0.1in} z\in\overline{\D}.
\end{equation} 
The proof strategy is as follows: first we show that $\{g_n\}_n$ form a normal family in $\D$. Then since the family $B_n :=\frac{f_n}{g_n}$ is holomorphic and uniformly bounded by $1$ on $\D $ (by estimate \eqref{fdomg}), Montel's theorem guarantees that $\{B_n\}$ is normal. Taken together, this implies that $f_n = B_ng_n$ also forms a normal family in $\D.$

\vspace{0.1in}

\noindent It remains to prove normality of $\{g_n\}$. Towards this, observe that the estimates \eqref{aprioribound} and \eqref{fdomg} show that
\begin{equation}\label{gsmallarea}
m\left(\Lambda_{g_n}\left(\frac12\right)\cap\D\right)\leq \frac{c_0}{n}, \hspace{0.1in} n\in\N.    
\end{equation}
Armed with this bound, we make the following 

\vspace{0.1in}

\noindent \textbf{Claim:} Let $c_1=\frac{100c_0}{c_{NPS}}$ and $r_n =\sqrt{\frac{c_1\log n}{n}}$. Then for all $n$ large enough we have
\[\Lambda_{g_n}\left(\frac12\right)\cap\D\left(0, 1-r_n\right)=\emptyset.\] Assume the claim for now. Then on any compact subset of $\D$ we would have $\vert g_n\vert\geq\frac12$ for all large enough $n$. This in turn implies by Montel that $\{\frac{1}{g_n}\}$ being uniformly bounded by $2$ on compacts, forms a normal family in $\D$. Since $g_n(0) = 1$ for all $n$, any limit $g$ of the sequence $\frac{1}{g_n}$ is non vanishing in $\D$ and satisfies $g(0) = 1$. From here it follows that $\{g_n\}$ is also normal. Now if the claim were not true, then we can find a sequence of natural numbers $N\to\infty,$ and points $z_0 = z_0(N)\in\D\left(0, 1-r_N\right)\cap\partial\Lambda_{g_N}(\frac12)$. Note that $d(z_0, \T)> r_N$. Consider the function $u_N(z) =\log|2g_N(z)|.$ Then since $g_N$ has all zeros in $\D^c$, we have that $u_N$ is harmonic in $B = B(z_0, r_N).$ Applying Theorem \ref{NPS2} (in conjunction with Remark \ref{NPSpolynom}) to $u_N$, we obtain
\[m\left(\Lambda_{g_N}\left(\frac12\right)\cap B\right)= m(\{u_N < 0\}\cap B)\geq c_{NPS}\frac{r_{N}^2}{\log 2N}= 100^2\frac{c_0\log N}{N\log 2N}> \frac{c_0}{N}\]
for large $N$, by our choice of $c_1$ and $r_N $. This contradicts the estimate \eqref{gsmallarea} and concludes the proof of the Theorem.
\end{proof}

\vspace{0.1in}

\noindent The normality of the $1-$ level minimizers with all zeros on $\T$ (denoted simply by $p_n$) is more delicate and we were unable to prove it one way or the other. However we will make some observations here that sheds more light on consequences of normality. We refer the reader to the overview section for notations that are used below. Suppose one could show that for some $M\in\N$, the family $\{p_n^{k_{n}}\}_n$ is \emph{not normal} in any neighborhood the origin, where $k_{n} =  \lfloor e^{(\log n)^M}\rfloor$. Then following the same proof of lower bound as in the case for $t>1$, we would be able to improve the lower bound in Theorem \ref{thm:arealevel1} to order $\frac{1}{\log\log n}$ (at least for infinitely many natural numbers $n$). 

\vspace{0.1in}

\noindent On the other hand, if for every $M\in\N,$ the family $\{p_n^{k_{n}}\}_n$ is normal in $r\D$, for some $r\in (0, 1)$, we are able to show that the $\Lambda_{p_n}$ have ``unstable" behavior for their areas in that a small increase in the level increases the area by a lot. More specifically, $m(\{\vert p_n\vert < 1\})\to 0$ as $n\to\infty$, whereas $m(\{\vert p_n\vert < 1 + \frac{2}{k_{n}}\})\geq c_0(r)>0$ for all sufficiently large $n$ (depending on $M$). Indeed, suppose that for $M\in\N,$ the family $\{p_n^{k_n}\}$ was normal in $r\D$ for an $r\in (0, 1)$. We note that since all the zeros of $p_n$ are on $\T$, we have $\vert p_n(0)\vert = 1$. Since rotating all the zeros by a fixed angle does not alter the lemniscate areas, we may assume that $p_n(0) = 1$ without loss of generality. Using this fact, it is easily seen that normality of $\{p_n^{k_n}\}$ implies that $p_n^{k_n}\to 1$ uniformly
on compact subsets of $r\D$ (use $m(\{\vert p_n\vert < 1\})\to 0$). Differentiation gives that $k_n p_n^{k_n -1}p_n'\to 0$ uniformly on compacts of $r\D$, in particular on $\frac{r}{2}\D$. Combining this with $p_n^{k_n}\to 1$, we obtain that for large $n$, $\vert p_n'(z)\vert\leq \frac{2}{k_n}$ for all $z$ in $\frac{r}{2}\D$. This in turn yields
\begin{equation}\label{unstable}
\vert p_n(z) - p_n(0)\vert = \vert p_n(z) - 1\vert = \vert \int_{0}^{z}p_n' (w)dw\vert\leq \frac{2}{k_n}, \hspace{0.1in} z\in \frac{r}{2}\D
\end{equation}
The estimate \eqref{unstable} shows that $1 - \frac{2}{k_n}\leq\vert p_n(z)\vert \leq 1 + \frac{2}{k_n}$ for all $z\in \frac{r}{2}\D$. Remembering that $m(\{\vert p_n\vert < 1\})\to 0$, we see that the bounds in the previous line imply that 
\[m(\{\vert p_n\vert < 1 + \frac{2}{k_n}\})\geq m(\{1 <\vert p_n\vert < 1 + \frac{2}{k_n}\})\gtrsim \pi r^2/4,\] 
and this establishes the aforementioned unstable behavior of the areas.



\section{Numerics for the minimizing configuration of zeros}\label{sec:numerics}
The main unresolved issue is the gap between the $\frac{1}{\log \log n}$ upper bound and the $\frac{1}{\log n}$ lower bound in Theorem~\ref{thm:mainthmsimple}. 
While it seems prohibitively difficult to use numerics to distinguish between logarithmic and doubly-logarithmic rates (or something in between), numerical experiements can be used to probe the minimizing configuration of zeros for low-degree cases that may then reveal some structure allowing for exact or asymptotic calculations to resolve the area question.
As remarked in the introduction, $z^n-1$ is not the answer, although it is the answer to many other problems posed in \cite{EHP}. But we have shown that the zeros must equidistribute on $\mathbb T$. Thus, it is intriguing to know what the actual optimal configuration is! 
In this section we present numerical experiments up to degree $n=8$, arriving at conjectured zero configurations that achieve the minimal lemniscate area of $\kappa_n=\kappa_n(\mathbb T,1)$.

All our computations were done on Mathematica.  Whenever we say area of a configuration of $n$ points, we mean the area of the filled  lemniscate (of level-1, unless otherwise mentioned) of the monic polynomial with roots given by the configuration.  We tried two kinds of algorithms.
\begin{enumerate}
    \item  Fix $m\gg n$ and go over all $n$-element subsets of $S_m=\{e^{2\pi i\frac{k}{m}}\ : \ 0\le k\le m-1\}$ and choose the subset that gives the smallest area. There are $\binom{m}{n}$ subsets which makes it impossible to go over all of them except for small $m$ and $n$. Hence, to reduce complexity, we may consider only subsets with some symmetry (e.g., under conjugation) and those that contain $1$ (as rotation of all the points leads to the same result) etc.
    \item Start with an arbitrary configuration of roots $z_0,\ldots ,z_{n-1}$ in clockwise order with $z_0=1$. Run through $z_1,\ldots ,z_{n-1}$ ($z_0$ remains at $1$) and when it is the turn of $z_k$, consider the arc $[z_{k-1},z_{k+1}]$ and change $z_k$ to whatever $w$ in the arc minimizes the area (of course, we  discretize the arc and choose $w$ within that finite set). Keep cycling through the $z_k$s, till a stable solution is reached. This algorithm runs very fast (within 2 or 3 rounds), but it appears to get stuck in local minima that are not global minima, see Figure~\ref{Fig:localminimaforrecursivealgo}. 
\end{enumerate}
 As we know nothing about the local minima, we stick to the first algorithm henceforth.
\begin{figure}[h]
\includegraphics[scale=0.4]{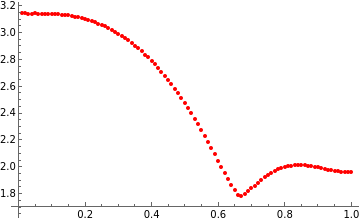}\;
\includegraphics[scale=0.4]{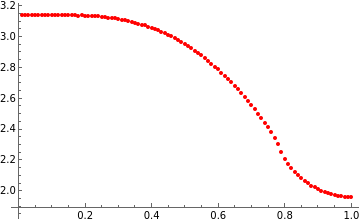}\;
\includegraphics[scale=0.4]{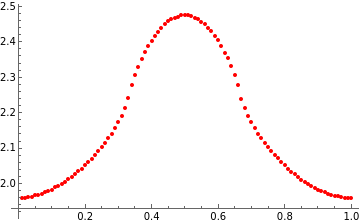}
\caption{Suppose $\arg z$ varies from $0$ to $1$. The lemniscate area is shown for these configurations with $n=3$: Left: $\{1,z,\bar{z}\}$, Middle: $\{1,1,z\}$, Right: $\{1,-1,z\}$. Conclusion: $\{1,1,-1\}$ is a local minimum in that changing any one root increases the lemniscate area. But it is not a global minimum.}\label{Fig:localminimaforrecursivealgo}
\end{figure}

\subsection*{Area computation} The key difficulty in carrying out the numerics is in the computation of the area of the filled lemniscate. We fixed on the following method (see Section~\ref{subsec:areacomputation} for comparison with other methods):

Randomly shift and rotate the triangular lattice $\Z+\Z e^{2\pi i/3}$,  and then scale it so that there are  $p$ (we took $p\approx 10^5$) lattice points in $2\D$. Call this collection of points $Q$.
For a region $R\subseteq 2\D$ that contains $N(R)$ points of $Q$, its area is taken to be  $4\pi N(R)/p$. Often we repeat this computation $T$ times (changing the  random shift and rotation in each round), and compute the mean and standard deviation to get an idea of the accuracy.

\subsection{First round of experiments}
Fix  $n$. Let $S=\{e^{i\pi j/m}\suchthat 1\le j\le m\}$, a set of $m$ equi-spaced points on the top half of the unit circle. For any $q,s,r$ such that $q+s+2r=n$, we take all  configurations with $q$ roots at $1$ and $s$ roots at $-1$ and $r$ roots in $S$ and their conjugates (so $S$ is symmetric about the $x$-axis)\footnote{We do not know if this symmetry must hold. However, experiments for very small $n$ without this assumption gave the same results.  Secondly, if $p_n$ is optimal for degree $n$, then $q_{2n}(z)=p_n(z)\overline{p_n(\overline{z})}$ is a polynomials of degree $2n$ whose roots are symmetric about the $x$-axis, and $\mbox{area}\{|q_n|<1\}\le 2\mbox{area}\{|p_n|<1\}$. By the bounds on $\kappa_n$ in Theorem~\ref{thm:arealevel1}, this means that $q_n$ also has area going to zero at the same rate as $p_n$. In other words, the symmetry restriction still gives us near-optimal configurations, if not optimal. }
. This gives $\binom{m}{r}$ configurations (for fixed $q,s,r$). 

We use $p\approx 100000$ points of the randomly shifted and rotated hexagonal lattice. We make $T$ trials (only changing the shift and rotation of the lattice points in each trial) where $T=20$ for $n\le 5$, $T=6$ for $n=6$. For $n=7,8$, we went by the pattern and assumed that $q\ge 2$ (which implies that $r\le 2$). The results of the experiment are shown in the table below. 

What is shown as the optimal configuration occurred in many trials, and in others, it was just off by a tiny bit. For example, for $n=4$, the output was  $\{1,e^{\pm 2\pi i \frac{1}{3}}\}$ in 11 trials and $\{1,e^{\pm 2\pi i \frac{74}{100}}\}$ in the remaining 9 trials. Similarly for $n=5$, the configuration $\{1,1,e^{\pm \pi i \frac{3}{5}},-1\}$ occurred all 20 times when we took $q=2,s=1$. The equivalent configuration $\{-1,-1,e^{\pm \pi i \frac{2}{5}},1\}$ all 20 times.

\begin{align*}
    \begin{array}{c|c|c}
   n & m & \mbox{Opt. config.} \\
   \hline
   3 & 99 & 1,e^{\pm 2\pi i\frac13}\\
   \hline
   4 & 100 & 1,1,e^{\pm \pi i \frac{3}{4}} \\
   \hline
   5 & 100 \mbox{ for }r\le 1\mbox{ and }\  50 \mbox{ for }r\ge 2 & 1,1,e^{\pm \pi i \frac{3}{5}},-1 \\
   \hline
   6 & 60 \mbox{ for }r\le 1 \mbox{ and }  24 \mbox{ for }r\ge 2 & 1,1,e^{\pm \pi i \frac{3}{6}},e^{\pm \pi i \frac{5}{6}} \\
   \hline
   7 & 70 (\mbox{ only } r\le 2) & 1,1,e^{\pm \pi i \frac{3}{7}},e^{\pm \pi i \frac{5}{7}},-1 \\
   \hline
   8 & 80 (\mbox{ only } r\le 2) & 1,1,1,e^{\pm \pi i \frac{4}{8}},e^{\pm \pi i \frac{6}{8}},-1 \\
   \hline
    \end{array}
\end{align*}

\begin{figure}[h] 
\includegraphics[width=0.3\textwidth]{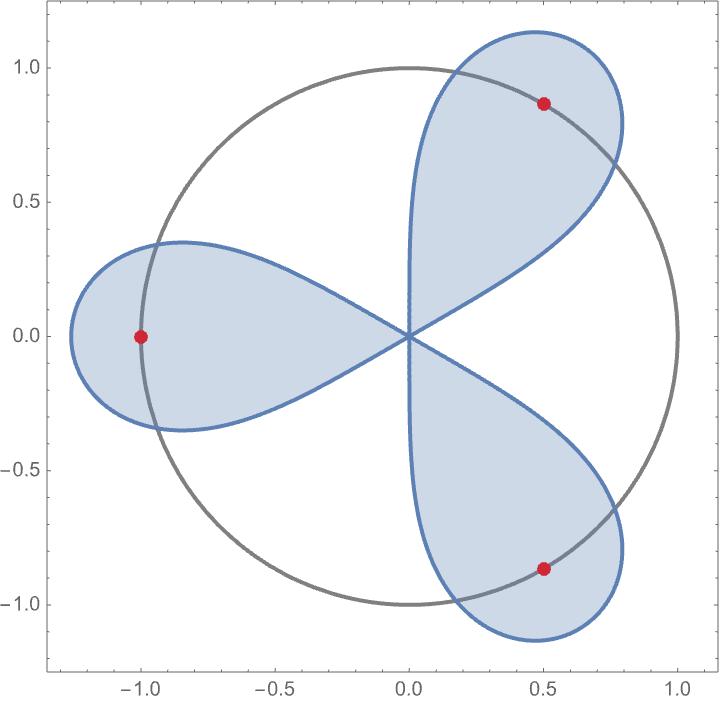}
\includegraphics[width=0.3\textwidth]{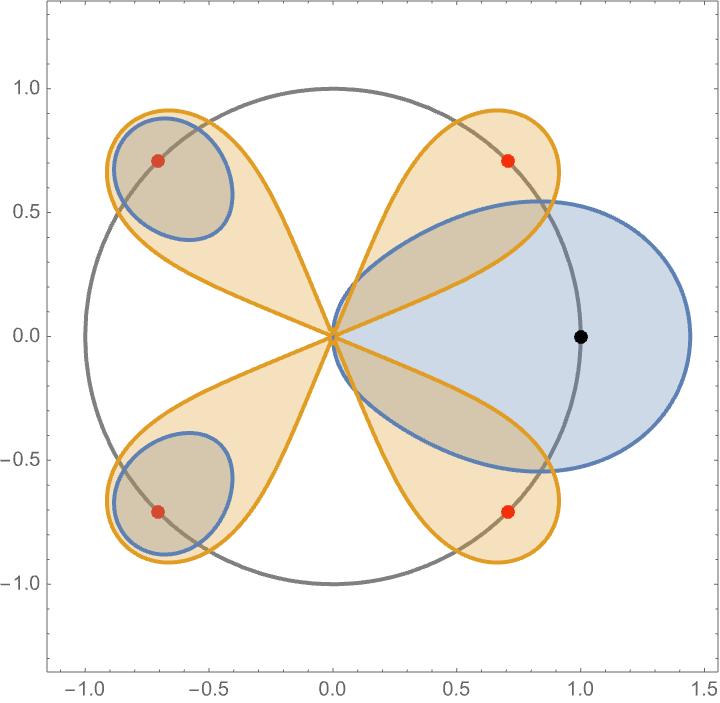}
\includegraphics[width=0.3\textwidth]{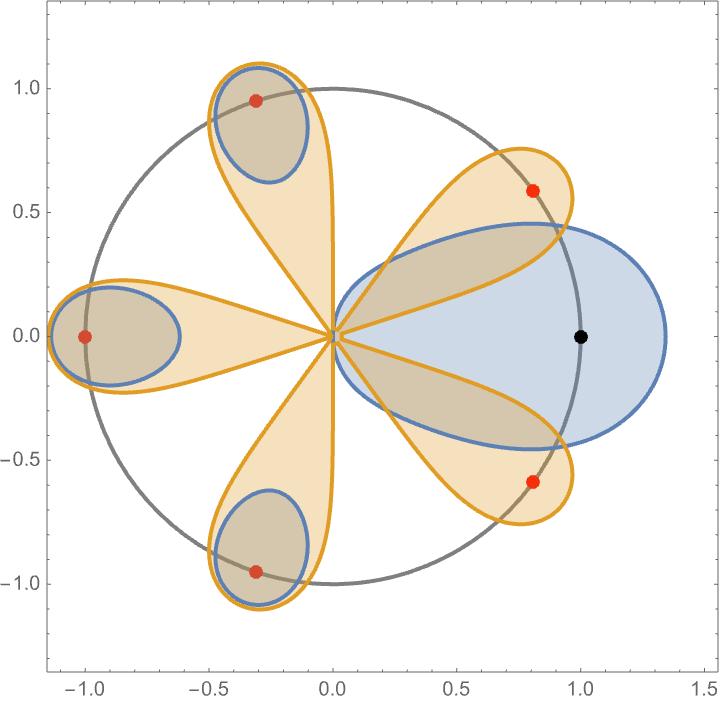} \\
\includegraphics[width=0.3\textwidth]{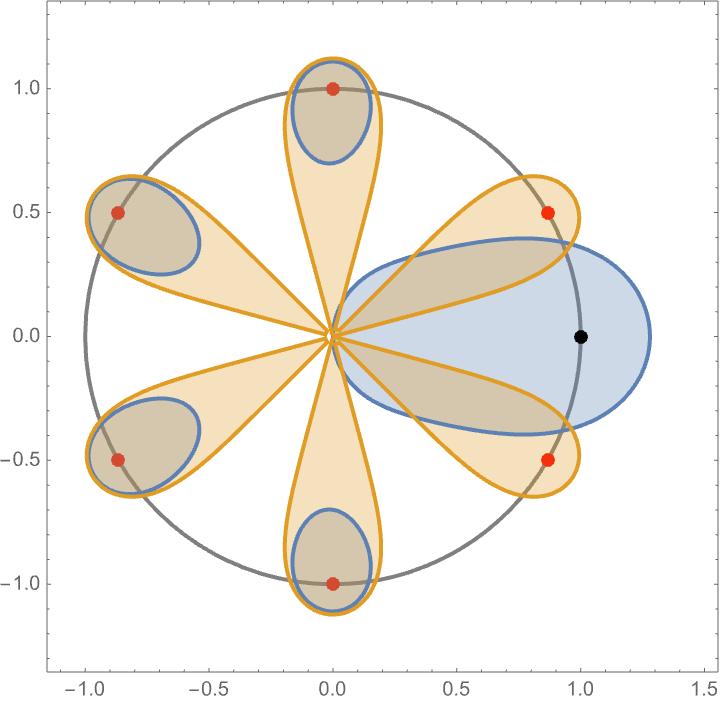}
\includegraphics[width=0.3\textwidth]{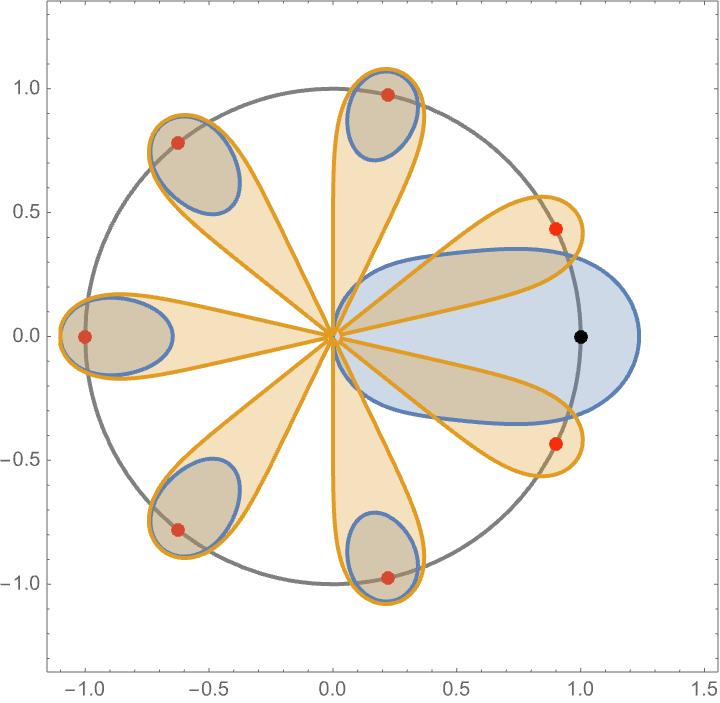}
\includegraphics[width=0.3\textwidth]{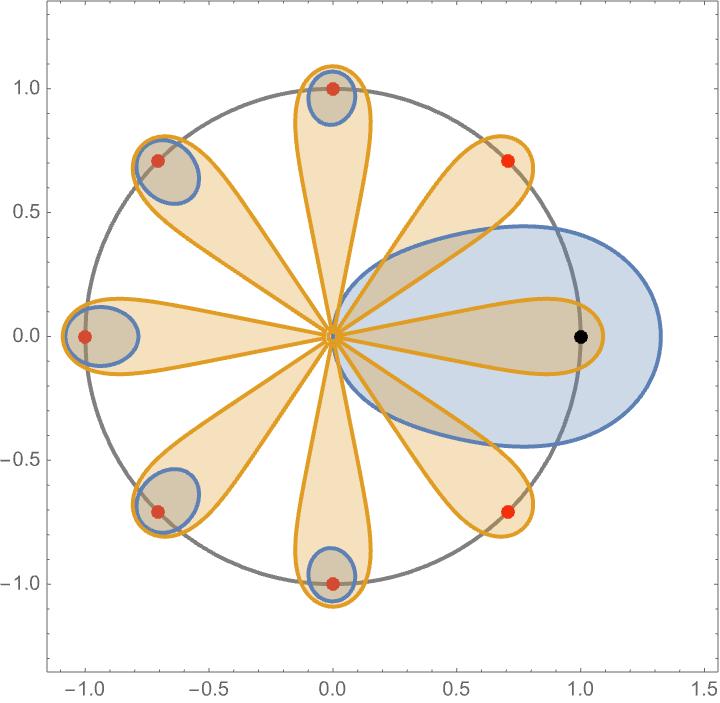}
\caption{The empirical minimizers for $n=3,4,5,6,7,8$.  For reference, we also plot the (possibly rotated) Erd\"os lemniscate of corresponding degree as well as the unit circle. For $n=3$ the empirical minimizer is the Erd\"os lemniscate.  In the other cases the minimizer consists of a large oval, formed by ``merging'' some petals in the corresponding Erd\"os lemniscate, along with small ovals within the remaining petals).}
\label{fig:lemniscatesupto8}
\end{figure}

In all cases the points in the optimal configuration are $2n$th roots of unity. In fact, 
   by a rotation, we can make all roots except at most one to be a $n$th root of unity, e.g., the configuration $\{1,1,e^{\pm \pi i \frac{3}{7}},e^{\pm \pi i \frac{5}{7}},-1\}$ becomes $\{e^{\pi i\frac{1}{7}},e^{\pi i\frac{1}{7}},e^{\pm \pi i \frac{4}{7}},e^{\pm \pi i \frac{6}{7}},e^{\pi i \frac{8}{7}}\}$ upon rotation by the angle $\frac{\pi}{7}$. Similar consideration applies to the configuration obtained for $4\le n\le 7$..

   For $1\le h\le n$, let  $\mathcal C_{n,h}$ be the configuration which is got by starting with the $n$th roots of unity and replacing the   $h$ consecutive roots $e^{2\pi i k/n}$, $0\le k\le h-1$, by one root of multiplicity $h$ placed at $e^{\pi i \frac{h-1}{n}}$ (mid-point of the arc containing the replaced roots). The remaining $n-h$ roots remain undisturbed. 

   With this notation, the minimizer configurations obtained numerically above are $\mathcal C_{2,1}$, $\mathcal C_{3,1}$, $\mathcal C_{4,2}$, $\mathcal C_{5,2}$, $\mathcal C_{6,2}$, $\mathcal C_{7,2}$,$\mathcal C_{8,3}$. See Figure~\ref{fig:lemniscatesupto8} for a visual comparison of these lemniscates to the corresponding $\mathcal C_{n,1}$ (the lemniscate of $z^n-1$).

{\em Assuming} that the minimizing configuration is among the $2n$th roots of unity, we  could proceed to somewhat higher values of $n$, as it is essentially the same experiment as above, but with $m=2n$, which is smaller than the $m$ we took in the first round. These experiments again gave the optimal configuration as $\mathcal C_{n,h}$ for some $h$.

\subsection{On the configurations  $\mathcal C_{n,h}$}  It is tempting to think that $\mathcal C_{n,h}$ is the optimal configuration for some $h=h(n)$. However, we can prove that these cannot be the optimal for large $n$, as their areas cannot go to zero as $n\to \infty$. 

\begin{claim}
Uniformly over $n\ge 3$ and $1\le h\le n-1$, the area of the configuration    $\mathcal C_{n,h}$ is bounded below by a positive constant. 
\end{claim}
\begin{proof}
For notational simplicity, we only prove the statement for odd $h$. The proof for the case $h$ even proceeds along similar lines. Fix $0<\alpha<\frac{\pi}{2}$ and observe that
\[
\frac{|z-1|^2}{|(z-e^{i\alpha})(z-e^{-i\alpha})|}\le 1 \Leftrightarrow (1+r^2)\cos \theta \ge r(1+\cos \alpha)
\]
where $z=re^{i\theta}$. This region is decreasing in $\alpha$ and converges to $A=\{z\suchthat (1+r^2)\cos \theta\ge 2r\}$. This set has positive area, as does  $A_{\eps}=\{z\suchthat (1+r^2)\cos \theta\ge (2-\eps) r\}$ if $\eps$ is small.

The polynomial corresponding to $\mathcal C_{n,2h+1}$ is 
\[
Q_{n}(z)=(z^n-1)\prod_{k=1}^h\frac{(z-1)^2}{(z-e^{2\pi i k/n})(z-e^{-2\pi i k/n})}.
\]
If $z\in A_{\eps}$, then  each of the $h$ factors is bounded above by $1-\delta$ for some $\delta=\delta(\eps)$. If further $r<1-\delta$, then $|z^n-1|<1+r^n<1+\delta$ if $n$ is large. For $h\ge 1$, this shows that $|Q_{n}(z)|<1$ for $z\in A_{\eps}\cap (1-\eps)\D$ which has constant area. For $h=0$, we have the Erd\"{o}s lemniscate $|z^n-1|\le 1$, which is easily seen to have constant area.
\end{proof}

\para{Conclusions} We conjecture that the optimal configuration is to be found among the $2n$th roots of unity. If so, since $\mathcal C_{n,h}$ is not the optimal solution, it seems likely that there are several roots with multiplicities higher than $1$. It is a numerical challenge worthy of pursuit to extend these experiments to higher degree $n$ until a clear pattern emerges.

\subsection{Methods for computing areas}\label{subsec:areacomputation} To compare different methods, we  tested them on the lemniscate $\{z : |z^n-1|<1\}$ which has area equal to
\begin{align*}
R(n)=2^{\frac{2}{n}}\frac{\sqrt{\pi}\Gamma(\frac12+\frac{1}{n})}{2\Gamma(1+\frac{1}{n})}.
\end{align*}
Consider
\begin{itemize}
    \item $R_1(n)$  based on the RegionMeasure command in Mathematica.
    \item $R_2(n)$ based on summing over $100651$ points of $L\cap 2\D$, where $L$ is the randomly shifted and rotated square lattice.
    \item $R_3(n)$ based on summing over $100668$ points of $L\cap 2\D$, where $L$ is the randomly shifted and rotated triangular  lattice.
    \item $R_4(n)$ based on summing over $100000$ points sampled uniformly at random from $2\D$.
\end{itemize}
The results of one experiment are summarized in the following table.
\begin{align*}
\begin{array}{c|c||c|c|c|c|}
n &R(n) & R_1(n)-R(n) & R_2(n)-R(n) & R_3(n)-R(n) & R_4(n)-R(n) \\
\hline 
2 & 2.	& -0.009544	& 0.000611 &	-0.000101&	0.010745\\
3 & 1.77829&	-0.010207&	-0.000161&	-0.000461&	0.005257\\
4 & 1.69443&	-0.01106&	-0.000074&	0.000763&	-0.010533\\
5 & 1.65321&	-0.00884&	-0.00168&	0.000288&	-0.005883\\
6 & 1.62978&	-0.009519&	-0.001221&	-0.000123&	-0.008715\\
7 & 1.61514&	-0.010185&	0.000304&	0.000281&	0.013711\\
8 & 1.60538&	-0.081407&	-0.000291&	-0.000687&	-0.000147\\
9 & 1.59853&	-0.081346&	0.000066&	0.000046&	0.002431\\
10 & 1.59353&	-0.047055&	0.000813&	-0.00033&	0.023633\\
\hline 
\end{array}
\end{align*}
From this table, it is evident that $R_{3}$ performs better than the others.


\subsection*{Acknowledgments}
The first named author is partly  supported by the DST FIST program-2021[TPN-700661]. The second-named author acknowledges support from the Simons Foundation (grant 712397).

\bibliographystyle{abbrv}
\bibliography{ref}

\end{document}